\documentclass[reqno,10pt]{amsart}
\usepackage{graphicx}

\usepackage{amsmath,amssymb,amsfonts,amsthm}
\usepackage{mathtools}
\usepackage[super]{nth}

\usepackage{stackengine}
\newcommand{\subsetdots}{
	\mathrel{\stackon[1.5pt]{$\subset$}{{..}}}
}

\usepackage[mathscr]{eucal}
\usepackage{bbm}
\usepackage{array}
\usepackage{mathdots}
\usepackage{ytableau}

\usepackage{hyperref}
\newcommand\myshade{85}
\colorlet{mylinkcolor}{red}
\colorlet{myurlcolor}{purple}
\colorlet{mycitecolor}{blue}
\hypersetup{
	linkcolor  = mylinkcolor,
	citecolor  = mycitecolor,
	urlcolor   = myurlcolor!\myshade!black,
	colorlinks = true,
}

\usepackage{tikz}

\usetikzlibrary{arrows,math,calc}
\usetikzlibrary{shapes.geometric}
\usetikzlibrary{patterns}
\usepackage[all,cmtip]{xy}

\usepackage{enumerate}
\usepackage[shortlabels]{enumitem}

\usepackage{subcaption}

\usepackage{etoolbox}
\patchcmd{\section}{\scshape}{\bfseries}{}{}
\makeatletter
\renewcommand{\@secnumfont}{\bfseries}
\makeatother

\theoremstyle{definition}
\newtheorem{thm}{Theorem}
\newtheorem*{thm*}{Theorem}
\newtheorem{prop}[thm]{Proposition} 
\newtheorem{lem}[thm]{Lemma}
\newtheorem{defi}[thm]{Definition} 
\newtheorem{exam}[thm]{Example} 
\newtheorem{nota}[thm]{Notation} 

\newtheorem{rem}[thm]{Remark}

\newtheorem*{quest*}{Question}

\usepackage{bbm}

\newcommand{\vac}{\mathbbm{1}}

\newcommand{\la}[1]{\mathfrak{#1}}
\newcommand{\gr}{\mathrm{gr}}
\newcommand{\ring}[1]{\mathsf{#1}}
\newcommand{\ala}[1]{\widehat{\mathfrak{#1}}}
\newcommand{\VOA}[1]{\mathsf{#1}}

\usepackage{xspace}
\newcommand{\hp}{Hilbert--Poincar\'e\xspace}

\newcommand{\ZZ}{\mathbb{Z}}
\newcommand{\NN}{\mathbb{N}}

\newcommand{\CC}{\mathbb{C}}

\newcommand{\sC}{\mathscr{C}}
\newcommand{\sG}{\mathscr{G}}
\newcommand{\sD}{\mathscr{D}}
\newcommand{\sK}{\mathscr{K}}
\newcommand{\sP}{\mathscr{P}}
\newcommand{\sF}{\mathscr{F}}

\newcommand{\sX}{\mathscr{X}}

\DeclareMathOperator{\wt}{wt}

\newcommand{\lietype}[1]{\mathrm{#1}}

\newcommand{\lt}{\ell t}
\newcommand{\LT}{\mathrm{LT}}
\newcommand{\sh}{\mathrm{sh}}
\newcommand{\tr}{\mathsf{Tr}}

\newcommand{\circnum}[1]{
	\tikz[baseline]{
		\node[ellipse,draw,inner sep=1pt,minimum width=1em, minimum height=1em]  {$\scriptstyle{#1}$};}}

\newcommand{\ord}[2]{\underbracket{#1}_{\circnum{#2}}}

\title{Classical freeness of $\widehat{\mathfrak{sl}}_n$ at level $1$ via combinatorics}
\author{Shashank Kanade}
\address{University of Denver, Denver, USA}
\email{shashank.kanade@du.edu}
\allowdisplaybreaks

\begin{document}
	
	\begin{abstract}
		We use a family of Rogers--Ramanujan-type combinatorial identities of Dousse--Konan involving coloured partitions to prove classical freeness of the simple vertex operator algebras based on $\widehat{\mathfrak{sl}}_n$ at level $1$. These identities are used to produce Gr\"obner bases for the relevant arc algebras. 
	\end{abstract}

	\maketitle

	\section{Introduction}
	
	For nearly half a century, it has been widely recognized that Rogers--Ramanujan-type identities concerning integer partitions govern the structure of representations of affine Lie algebras and, more broadly, of vertex operator algebras (henceforth VOAs). This started with the work of Lepowsky--Milne \cite{LepMil} who first showed that the product sides in Rogers--Ramanujan identities are related to the principally specialized characters of level $3$ standard modules of $\ala{sl}_2$. Then, a few years later, Lepowsky--Wilson \cite{LepWil-pnasLieRR}, \cite{LepWil-pnasNewAlg}, \cite{LepWil-structI} constructed bases of relevant substructures of these modules enumerated by the ``difference $2$'' partitions counted by the sum-sides of Rogers--Ramanujan identities. Lepowsky--Wilson's work involved an independent discovery (and use) of vertex operators in the mathematical literature.
	
	In the last several decades, many new such identities have been found using Lepowsky--Wilson's framework, the first of which were Capparelli's identities \cite{Cap-id} related to $\lietype{A}_2^{(2)}$. Fascinating discoveries of intricate new identities continue to be made. A brief (and necessarily incomplete) summary of the developments pertaining roughly to the last decade could be found in \cite{Kan-survey}.
	
	On the other hand, knowledge of such identities can help establish structural properties of VOAs. One such structural property which has recently risen in prominence was first identified by Arakawa--Moreau \cite{AraMor-arc} and was called classical freeness by van Ekeren--Heluani \cite{vanEkeHel-chiralhom}: A VOA is called classically free if it is the quantization of the arc algebra of its associated scheme.

	It is then remarkable that this algebro-geometric property of classical freeness which has deep technical implications (for instance, to computations of chiral homology of elliptic curves as studied in \cite{vanEkeHel-chiralhom}) essentially reduces to purely combinatorial Rogers--Ramanujan-type identities. Indeed, this property is established if one can show that the Hilbert series of the arc algebra equals the character of the VOA. For a broad class of VOAs  including the ones considered here, the former is the generating function of coloured partitions with some frequency restrictions, and the latter is usually expressed in terms of theta functions of lattices. The equality of these two is precisely a Rogers--Ramanujan-type identity.
	
	Classical freeness (or the quantification of the failure thereof) in the following examples has been established combinatorially using Rogers--Ramanujan-type or $q$-series identities. We denote by $\VOA{L}_k(\la{g})$ the simple VOA based on a (simple) Lie algebra $\la{g}$ at level $k (\neq -h^\vee)\in\CC$.
	\begin{enumerate}[leftmargin=*]
		\item Virasoro minimal models at boundary admissible parameters $(2,2k+3)$ are classically free due to the Andrews--Gordon identities, see the work of Bruscheck--Mourtada--Schepers \cite{BruMouSch-rr} and van Ekeren--Heluani \cite{vanEkeHel-chiralhom}. There is a different circle of ideas relating the arc algebra here with the principal subspace of $\VOA{L}_k(\la{sl}_2)$ and this leads to another proof of classical freeness based on \cite{CapLepMil-rr}, \cite{CapLepMil-ag}, \cite{CalLepMil-rr}, \cite{CalLepMil-ag}.			
		
		\item Rational Virasoro minimal models beyond boundary admissible parameters are easily seen to \emph{not} be classically free \cite{vanEkeHel-chiralhom}. Characterizing the failure at parameters $(3,4)$ is again intimately tied with an exotic Rogers--Ramanujan-type identity found and proved in  \cite{AndvanEkeHel} by Andrews--van Ekeren--Heluani.
		
		\item The VOAs $\VOA{L}_k(\la{sl}_2)$ for $k\in\ZZ_{\geq 0}$ are classically free  \cite{vanEkeHel-chiralhom} due to the coloured partitions identities of Meurman--Primc \cite{MeuPri-annfields}.
		
		\item Linshaw--Song \cite{LinSong-invsympl} proved classical freeness of $\VOA{L}_k(\la{sp}_{2n})$ for $k\in\ZZ_{\geq 0}$ using invariant-theoretic methods.
		For $k=1$, the classical freeness here is also implied by the coloured partition identities
		due to Primc--\v{S}iki\'c \cite{PriSik-1}, \cite{PriSik-2}, \cite{PriSik-3}  and Primc--Trup\v{c}evi\'c \cite{PriTru} (see also \cite{CMPP}).  
		The proof of how these combinatorial identities are used for classical freeness at level $1$ is given by Li in \cite{Li-remarks}. The combinatorial level $1$ identities that are required here are also proved independently by Dousse--Konan \cite{DouKon-sp2n} and in the principal gradation by Russell \cite{Rus-CMPP}.
		
		\item Principal subspaces for $\VOA{L}_1(\la{sl}_n)$ (and a few other examples in other types) are classically free due to Li--Milas \cite{LiMil} where 
		certain quantum dilogarithm $q$-series identities  are used.
		Principal subspaces for certain graph vertex algebras are classically free due to Li \cite{Li-remarks}.
		\item Classical freeness of super VOAs based on $N=1$ minimal models at parameters $(2,4k)$ was established combinatorially by Li \cite{Li-remarks}.
	\end{enumerate} 
	It is important to remark here that Meurman--Primc's work \cite{MeuPri-annfields} on  $\VOA{L}_k(\la{sl}_2)$ ($k\in\ZZ_{\geq 0}$) achieves far more than what we have mentioned above. They actually produce and prove non-commutative Gr\"obner bases, and after this, one does the comparatively  easier task of transferring these to Gr\"obner bases of (commutative) arc algebras. Meurman--Primc extended their approach to $\VOA{L}_1(\la{sl}_3)$ in \cite{MeuPri-sl3}, and the work \cite{PriSik-1}, \cite{PriSik-2}, \cite{PriSik-3}, \cite{PriTru} on $\VOA{L}_k(\la{sp}_{2n})$ recalled above also follows the same framework.
	We expect that a similar transfer of the non-commutative bases of Primc--\v{S}iki\'c \cite{PriSik-2} and Primc--Trup\v{c}evi\'c \cite{PriTru}  to arc algebras should provide a new proof of classical freeness at higher levels for $\VOA{L}_k(\la{sp}_{2n})$.
	
	There do exist other approaches to classical freeness. Most prominently, the invariant-theoretic approach of Linshaw--Song \cite{LinSong-cosets}, \cite{LinSong-invsympl} culminating in the work of Creutzig--Linshaw--Song \cite{CreLinSong} establishes this property for super VOAs $\VOA{L}_n(\la{osp}(m|2r))$ with $-\frac{m}{2}+r+n+1>0$.  
	
	Despite all of this, classical freeness remains ``\emph{quite mysterious}'' \cite{CreLinSong}, and it is quite hard to tell if a given arbitrary VOA (especially, a rational VOA) is classically free. Several VOAs are known  \emph{not} to be classically free -- Virasoro minimal models beyond boundary parameters (as recalled above) \cite{vanEkeHel-chiralhom}, \cite{Sal-vir}, simple $\mathcal{W}_3$ algebra at central charge $c=-2$ \cite{AraLin-singular}, etc.
	Even among $\VOA{L}_1(\la{g})$, classical-freeness is quite puzzling. While $\VOA{L}_1(\la{sp}_{2n})$ case is classically free (see above), it has been proven recently by Song--Zeng \cite{SongZeng-E8} that $\VOA{L}_1(\lietype{E}_8)$ is not, providing a counter-example to a belief in the VOA literature that $\VOA{L}_1(\la{g})$ is classically free for all simple $\la{g}$ (cf.\ \cite{Fei-pbw}).
	
	In this paper, we establish classical freeness for $\VOA{L}_1(\la{sl}_n)$ (Theorem \ref{thm:main} below)  using a family of combinatorial identities of Dousse--Konan \cite{DouKon-slnI}, \cite{DouKon-slnII}. These identities were in turn found using sophisticated combinatorial constructions starting from the ``KMN$^2$'' \cite{KKMMNN} character formula involving crystal bases. This approach has the advantage that we are simultaneously able to produce a Gr\"obner basis for the arc algebra. Additionally, the link to Dousse--Konan's work opens up further natural avenues to exploit crystal character formulas in other types or at higher levels.

	The strategy for combinatorially establishing the classical freeness is straightforward to explain. We start running Buchberger's algorithm on the arc algebra. The natural starting generating set is formed by a specifically chosen basis of quadratic elements in the singular space and their derivatives.  Section \ref{sec:quad} is devoted to this calculation. In Section \ref{sec:cubic},  we then find that modulo this generating set, several $S$-polynomials do not reduce to $0$ (this was first observed in the $\VOA{L}_1(\la{sl}_3)$ level $1$ case  by Meurman--Primc \cite{MeuPri-sl3}). We enlarge the generating set by including these reductions whose leading terms are cubics.  
	At this point, the Hilbert series of leading terms found thus far matches the conditions in Dousse--Konan's identities \cite{DouKon-slnI}, however, this also requires quite a bit of work detailed in Section \ref{sec:dk}. The primary reason that this is not immediate is that the combinatorial framework used by Dousse--Konan rests on ``difference conditions'' while we need a description in terms of ``forbidden frequencies'' adapted to the leading terms in the (eventual) Gr\"obner bases. Once this is done, Dousse--Konan's  family of identities immediately gives us that the Hilbert series of the leading terms found thus far is exactly the character of $\VOA{L}_1(\la{sl}_n)$. This simultaneously establishes classical freeness and also proves that Buchburger's algorithm has terminated, i.e., that we have found a Gr\"obner basis for the arc algebra; see our main Theorem \ref{thm:main} presented in Section \ref{sec:main}. A few questions arising from this strategy are collected in Section \ref{sec:questions}.

	Despite the fact that this is a paper on VOAs, almost all of the arguments  presented here are purely Lie-theoretic, commutative-algebraic, or combinatorial in nature. Nonetheless, we shall need to recall definitions regarding VOAs in Section \ref{sec:voa} to set the stage and to introduce the objects involved. Definitions and results for Gr\"obner bases are recalled in Section \ref{sec:grobner}. The degree reverse lexicographic monomial ordering we use is clarified in Section \ref{sec:setup}. Once this is done, the bulk of the computations  follow as explained in the previous paragraph. 
	
	\subsection*{Acknowledgments} I am infinitely grateful to Mirko Primc for drawing my attention to the work of Dousse--Konan \cite{DouKon-slnI} and explaining how it relates to the work of Meurman--Primc \cite{MeuPri-sl3} in the $\VOA{L}_1(\la{sl}_3)$ case. I sincerely thank Jehanne Dousse for discussions on the work of Dousse--Konan \cite{DouKon-slnI} and \cite{DouKon-slnII}. 	
	It is a privilege to thank Andy Linshaw and Juan Villareal for discussions pertaining to classical freeness. 
	I am grateful to Bailin Song and  Xianlong Zeng for sharing with me a preprint of their recent work \cite{SongZeng-E8}.
	Finally, I acknowledge financial support from the Simons Foundation via the Travel Support for Mathematicians grant.

	\section{Classical freeness for vertex operator algebras}
	\label{sec:voa}
	In this section, we recall the notions of vertex operator algebras, Li's filtration, Zhu's $C_2$ algebra, its arc algebra, and ultimately, classical freeness.

	Let $\VOA{V}=\oplus_{n\in\ZZ}\VOA{V}_n$ be a $\ZZ$-graded vector space with finite-dimensional graded components with distinguished elements $\vac\in\VOA{V}_0$, $\omega\in\VOA{V}_2$, called the vacuum and the conformal vector, respectively.
	
	Suppose that $\VOA{V}$ is equipped with a state-field correspondence, i.e., a map 
	$$Y: \VOA{V}\rightarrow (\mathrm{End}(\VOA{V}))[[x,x^{-1}]],$$
	whose value on individual elements $v\in\VOA{V}$ is denoted as:
	\begin{align*}
		Y(v,x)=\sum_{n\in\ZZ}v_{(n)}x^{-n-1}\quad (v_{(n)}\in\mathrm{End}(\VOA{V})).
	\end{align*}
	The coefficients of $Y(\omega,x)$ are enumerated differently, and are given special names:
	\begin{align*}
		Y(\omega,x)=\sum_{n\in\ZZ}L_nx^{-n-2},
	\end{align*}
	equivalently, $L_n = \omega_{(n+1)}$.
	
	$\VOA{V}$ equipped with $Y$ is called a vertex operator algebra of central charge $c\in\CC$ if this data satisfies the following properties:
	\begin{enumerate}
		\item Lower truncation: For all $v,w\in\VOA{V}$, $Y(v,x)w\in \VOA{V}((x))$.
		\item Vacuum and creation: $Y(\vac,x)v=v$ for all $v\in \VOA{V}$, $Y(v,x)\vac \in  v + x\VOA{V}[[x]]$.
		\item Locality: For all $v,w\in \VOA{V}$, there exists an $N\in \NN$ such that $$(x_1-x_2)^N[Y(v,x_1),Y(w,x_2)]=0.$$
		\item $L_{-1}$ derivative property: For all $v\in\VOA{V}$, $$Y(L_{-1}v,x)=[L_{-1},Y(v,x)]=\dfrac{d}{dx}Y(v,x).$$
		\item $L_0$ grading property (conformal weights): For all $v\in \VOA{V}_n$, $L_0v=nv$.
		\item Conformal vector: The modes $L_n$ provide a representation of the Virasoro algebra, i.e.,  for all $m,n\in\ZZ$,
		$$[L_m,L_n]=(n-m)L_{m+n}+\dfrac{n^3-n}{12}\delta_{n+m=0}c.$$
	\end{enumerate}
	For alternate axiomatic definitions of VOAs, see \cite{LepLi}.
	
	Next, due to \cite{Li-abelianizing}, there is a  decreasing filtration $F_\bullet\VOA{V}$ on $\VOA{V}$ such that the associated graded $\gr_{F}\VOA{V}$ can be equipped with a canonical structure of a commutative, differentially-graded ring. Define:
	\begin{align*}
		F_p(\VOA{V})=\mathrm{Span}_\CC\{ u^{(1)}_{-k_1-1}\cdots u^{(t)}_{-k_t-1}\vac\,|\, t\geq 1, k_1,\cdots,k_r\geq 0, k_1+\cdots+k_r\geq p    \}.
	\end{align*}
	It is easy to see that
	\begin{align*}
		\VOA{V}=F_0(\VOA{V})\supseteq F_1(\VOA{V})\supseteq \cdots,
	\end{align*}
	and we let:
	\begin{align*}
		\gr_F \VOA{V}=\bigoplus_{p\geq 0} F_p(\VOA{V})/F_{p+1}(\VOA{V}).
	\end{align*}
	We let $\sigma_p$ be the corresponding principal symbol map, i.e., 
	$$\sigma_p: F_p(\VOA{V})\rightarrow F_p(\VOA{V})/F_{p+1}(\VOA{V})\hookrightarrow\gr_F(\VOA{V}).$$
	The associated graded $\gr_F(\VOA{V})$ equipped with the product
	$$\sigma_p(a)\sigma_q(b)=\sigma_{p+q}(a_{(-1)}b)$$
	for $a\in F_{p}\VOA{V}$, $b\in F_{q}\VOA{V}$ (and extended linearly) is a commutative ring with unity $\sigma_0(\vac)$. Further, the operator $L_{-1}$ gives rise to a canonical derivation of this ring:
	\begin{align*}
		\partial(\sigma_p(a))=\sigma_{p+1}(L_{-1}a).
	\end{align*}
	Thus, $\gr_F(\VOA{V})$ is a differential commutative ring.
	
	It is now clear that $F_0(\VOA{V})/F_1(\VOA{V})$ is a commutative subalgebra of $\gr_F(\VOA{V})$. It is in fact a Poisson algebra, with bracket defined by:
	\begin{align*}
		\{\sigma_0(a),\sigma_0(b)\}=\sigma_0(a_{(0)}b).
	\end{align*}
	This is precisely Zhu's $C_2$-algebra \cite{Zhu1996-modinv}, denoted as $R_{\VOA{V}}$:
	\begin{align}
		R_{\VOA{V}}=F_0(\VOA{V})/F_1(\VOA{V})\cong \VOA{V}/\mathrm{Span}\{a_{-2}b\,|\,a,b\in \VOA{V} \},
	\end{align}
	where the last isomorphism is as vector spaces.

	The first theorem is the following:
	\begin{thm}[{\cite[Cor.\ 4.3]{Li-abelianizing}}]
		\label{thm:JRVsurj-grV}
		As a differential algebra, $\gr_F(\VOA{V})$ is generated by $R_{\VOA{V}}$.
	\end{thm}
	
	It is now natural to ask the following question: how big is the resulting object if we attach a derivation to $R_{\VOA{V}}$ in the freest possible way? Clearly, due to the theorem just quoted, this object is at least as big as $\gr_F{\VOA{V}}$. $\VOA{V}$  is said to be classically free if this object is exactly as big as $\gr_F(\VOA{V})$. We now define this formally.
	
	\begin{defi}[{\cite[Prop.\ 2.16]{SongZeng-E8}}]
		Given a commutative algebra $R$ over $\CC$, there exists a unique up to isomorphism commutative differential algebra
		$J_{\infty}R$ over $\CC$ satisfying the following universal property. For any differential commutative algebra $A$ over $\CC$, we have:
		$$\hom_{\mathcal{D}}(J_{\infty}R,A)\cong \hom_{\mathcal{C}}(R,A),$$
		where $\mathcal{D}$ is the category of differential commutative algebras over $\CC$ and $\mathcal{C}$ is the category of commutative algebras over $\CC$. We call $J_{\infty}R$ the arc algebra of $R$.
	\end{defi}
	
	Theorem \ref{thm:JRVsurj-grV} implies that we have a natural surjection of differential commutative algebras:
	\begin{align}
		J_{\infty}R_{\VOA{V}}\twoheadrightarrow\gr_F(\VOA{V}).\label{eqn:JRgrVsurj}
	\end{align}
	
	\begin{defi}
		$\VOA{V}$ is called classically free if \eqref{eqn:JRgrVsurj} is an isomorphism.
	\end{defi}

	When $R\cong \CC[x_1,\cdots,x_N]/\langle f_1,\cdots, f_M\rangle$,  $J_{\infty}R$ maybe taken to be (see \cite{SongZeng-E8}):
	\begin{align}
		J_\infty R = \dfrac{\CC[\partial^jx_i\,|\, 1\leq i\leq N, j\geq 0]}{\langle \partial^jf_i\,|\,1\leq i\leq M, j\geq 0\rangle}.
		\label{eqn:JRaff}
	\end{align}

	While $\gr_{F}(\VOA{V})$ has a natural gradation by filtration degree, we need another gradation. It is a consequence of the axioms that each $F_p(\VOA{V})$ is also graded by conformal weight $L_0$. Consequently, $\gr_F(\VOA{V})$ and $R_{\VOA{V}}$ are also graded by $L_0$. Assigning the weight $1$ to $\partial$, $J_{\infty}R_{\VOA{V}}$ is also graded (we again call this the conformal grading), and the map \eqref{eqn:JRgrVsurj} is a grade-preserving surjection.
	Let us create the (graded) characters (or the Hilbert series) of the objects involved:
	\begin{align*}
		\chi_{X}(q)=\sum_{n\in\ZZ}\dim(X_{n})q^n,
	\end{align*}
	where $X\in\{J_{\infty}R_{\VOA{V}}, \gr_F(\VOA{V}), \VOA{V}\}$,  $X_{n}$ denotes the homogeneous space of $X$ with conformal weight $n$, and $q$ is a formal variable.

	It is known that 
	\begin{align}
		\chi_{\gr_{F}\VOA{V}}(q)=\chi_{\VOA{V}}(q).
		\label{eqn:chgr=chV}
	\end{align}
	In terms of characters, \eqref{eqn:JRgrVsurj} implies the coefficient-wise inequality
	\begin{align}
		\chi_{J_{\infty}R_{\VOA{V}}}(q)\geq \chi_{\VOA{V}}(q).
		\label{eqn:chJR>=chV}
	\end{align}
	A consequence of this is the following. 
	\begin{thm}
	\label{thm:charclassfree}		
	A vertex operator algebra $\VOA{V}$ is classically free iff
	\begin{align}
		\chi_{J_{\infty}R_{\VOA{V}}}(q)=\chi_{\VOA{V}}(q).
		\label{eqn:cfiffchareq}
	\end{align}
	\end{thm}

	Now we come to the class of vertex operator algebras of our interest.
	
	Let $\la{g}$ be a simple finite-dimensional Lie algebra over $\CC$, with a non-zero invariant symmetric bilinear form $\langle\cdot,\cdot\rangle$ normalized so that the longest root has squared length $2$. We will work with $\la{g}=\la{sl}_n$ with $\langle X,Y\rangle = \mathrm{Tr}(XY)$ throughout.
	From this we can form the affine Lie algebra:
	$$\ala{g}=\la{g}\otimes \CC[t,t^{-1}]\oplus \CC K,$$
	where $K$ is central and the brackets are defined as:
	$$[X\otimes t^n,Y\otimes t^m]=[X,Y]\otimes t^{n+m}+n\langle X,Y\rangle \delta_{n+m=0}K.$$
	Given a $k\in\CC$ (called the level), form the generalized Verma module:
	$$\VOA{V}_k(\la{g})=\mathrm{Ind}_{\la{g}[t]\oplus \CC K }^{\ala{g}}\CC_k,$$
	where $K$ acts on $\CC_k$ by $k$ and $\la{g}[t]$ acts trivially. As a $\ala{g}$ module, $\VOA{V}_k(\la{g})$ has a unique maximal ideal $\VOA{I}_k(\la{g})$ that does not intersect $\CC_k$, and the quotient is denoted as:
	$$\VOA{L}_k(\la{g})=\VOA{V}_k(\la{g})/\VOA{I}_k(\la{g}).$$
	Very importantly, for all $k\neq -h^\vee$, $\VOA{V}_k(\la{g})$ has a canonical vertex operator algebra structure, which descends to a vertex operator algebra structure on $\VOA{L}_k(\la{g})$
	as described in \cite{LepLi}. This fact was first established by Frenkel--Zhu \cite{FreZhu}, but an alternate proof could be also found in \cite{MeuPri-annfields}.

	We now describe Zhu's $C_2$ algebra for $R_{\VOA{L}_k(\la{g})}$. We know that the symmetric algebra $\mathrm{Sym}(\la{g})$ has a natural adjoint action of $\la{g}$, which we denote by $\mathrm{ad}$. Let $\theta$ be the highest root of $\la{g}$ and $e_{\theta}$ the corresponding root vector.
	\begin{thm} We have (see for instance \cite{GabGan}, \cite{FeiFeiLit-c2}):
	\begin{align}
		R_{\VOA{L}_k(\la{g})}=\dfrac{\mathrm{Sym}(\la{g})}{\langle
			\mathrm{ad}(\la{U}(\la{g}))\cdot e_{\theta}^{k+1} \rangle }.
			\label{eqn:RLgk}
	\end{align}
	\end{thm}
	As a $\la{g}$-module, $\mathrm{ad}(\la{U}(\la{g}))\cdot e_{\theta}^{k+1}$ is finite-dimensional irreducible with highest weight $(k+1)\theta$.
	Thus, $J_{\infty}R_{\VOA{L}_k(\la{g})}$ is of the form \eqref{eqn:JRaff} with $\{x_i\}$ being a basis of $\la{g}$ and $\{f_i\}$ being a basis of $\mathrm{ad}(\la{U}(\la{g}))\cdot e_{\theta}^{k+1}$.

	\section{\texorpdfstring{Gr\"obner}{Groebner} bases}
	\label{sec:grobner}
	Let $\ring{R}=\CC[X^{(1)},X^{(2)},\cdots]$. We suppose that each $X^{(i)}$ is assigned a positive integral weight denoted by $\wt(X^{(i)})$ such that 
	\begin{align*}
		\wt(X^{(i)})&\leq \wt(X^{(i+1)}), \qquad \lim_{i\rightarrow\infty}\wt(X^{(i)})=\infty.
	\end{align*}
	The first will ensure that the grevlex order (defined below) is compatible with weights and the second will ensure that the Hilbert series of various objects are well-defined.
	
	An element of the sort 
	$$m=\prod_{i\geq 1} (X^{(i)})^{\alpha_i} \in \ring{R}$$
	where $\alpha_i\in\ZZ_{\geq 0}$ ($i\geq 1$) with finitely many $\alpha_i$ non-zero is called a \emph{monomial}.
	The weight of $m$ is defined as:
	\begin{align*}
		\wt(m)=\sum_{i\geq 1}\wt(X^{(i)})\alpha_i.
	\end{align*}
	It is convenient to express a monomial $a$ succinctly as:
	\begin{align*}
		m=\boldsymbol{X}^{\boldsymbol{\alpha}}.
	\end{align*}
	
	We now define the grevlex (degree reverse lexicographic) order. 
	\begin{defi}
	\label{def:grevlex}	
	We say that 
	\begin{align*}
		\prod_{i\geq 1} \boldsymbol{X}^{\boldsymbol{\alpha}}\succ
		\prod_{i\geq 1} \boldsymbol{X}^{\boldsymbol{\beta}}
	\end{align*}
	whenever
	\begin{enumerate}
		\item $\wt(\boldsymbol{X}^{\boldsymbol{\alpha}})>\wt(\boldsymbol{X}^{\boldsymbol{\beta}})$, or,
		\item $\wt(\boldsymbol{X}^{\boldsymbol{\alpha}})=\wt(\boldsymbol{X}^{\boldsymbol{\beta}})$ but the \emph{right-most} non-zero element of $\boldsymbol{\alpha}-\boldsymbol{\beta}$ is \emph{negative}.
	\end{enumerate}
	\end{defi}
	The following properties of this order are easy to prove:
	\begin{enumerate}
	\item $1\succeq \boldsymbol{X}^{\boldsymbol{\alpha}}$ with equality iff  $\alpha= \boldsymbol{0}$,
	\item $\boldsymbol{X}^{\boldsymbol{\alpha}}\succeq \boldsymbol{X}^{\boldsymbol{\beta}}$ implies 
	$\boldsymbol{X}^{\boldsymbol{\alpha}}\boldsymbol{X}^{\boldsymbol{\gamma}}\succeq \boldsymbol{X}^{\boldsymbol{\beta}}\boldsymbol{X}^{\boldsymbol{\gamma}}$.
	\end{enumerate}
	
	\begin{defi}
	\label{def:term}
	Given $f\in \ring{R}\backslash \{0\}$,
	a \emph{term} of $f$ is a monomial that appears in $f$ along with its coefficient.
	The leading term $\lt(f)$ is the term corresponding to the highest monomial of $f$.
	\end{defi}
	
	\begin{defi}
	The \emph{shape} of a monomial is defined as 
	\begin{align*}
		\sh\left(\boldsymbol{X}^{\boldsymbol{\alpha}}\right)
		=\prod_{i\geq 1}  w_{\wt(X^{(i)})}^{\alpha_i}\in\CC[w_1,w_2,\cdots]
	\end{align*}
	\end{defi}
	
	The importance of this definition is this: if we assume a grevlex order on $\CC[w_{1},w_{2},\cdots]$ with $\wt(w_{i})=i$, it is clear that for two monomials $a$, $b$ of $\ring{R}$, $\sh(a)\succ\sh(b)$ implies $a\succ b$.

	Given a subset $G\subset \ring{R}\backslash\{0\}$ we let $\LT(G)$ be the ideal generated by the leading terms of elements of $G$.

	Let $\ring{I}$ be any ideal of $\ring{R}$ that is generated by weight homogeneous elements. Then,  $\ring{R}/\ring{I}$ is also graded by weight, with each weight space finite-dimensional.

	For a weight-homogeneous ideal $\ring{I}\subseteq \ring{R}$, and $q$ a formal variable, we define the \hp series to be:
	\begin{align*}
		H_{\ring{I}}(q)=\sum_{n\geq 0}(\ring{R}/\ring{I})_nq^n.
	\end{align*}
	
	We now have the following fundamental property:
	\begin{thm}[{\cite[Ch.\ 9 \S 3, Prop.\ 9]{CoxLitOSh}}]
		For a homogeneous ideal $\ring{I}\subseteq\ring{R}$, 
		\begin{align*}
			H_{\ring{I}}(q) = H_{\LT(\ring{I})}(q).
		\end{align*}
	\end{thm}

	\begin{defi}[{\cite[Ch.\ 2 \S 5, Def.\ 5]{CoxLitOSh}}]
		Let $\ring{I}$ be an ideal of $\ring{R}$ generated by 
		a set $G\subseteq \ring{I}\backslash\{0\}$. We say that $G$ is a Gr\"obner basis of $\ring{I}$ if 
		$\LT(G)=\LT(\ring{I})$.
	\end{defi}
	
	Given $f,g,h\in \ring{R}$ with $g\neq 0$, we say that $f$ reduces to $h$ modulo $\{g\}$ if $\lt(g)$ divides a term $X$ in $f$ and
	$$h=f-\dfrac{X}{\lt(g)}g.$$
	We denote this as:
	\begin{align*}
		f\xrightarrow{\{g\}}h.
	\end{align*}
	Given $f,h\in\ring{R}$ and $G\subseteq \ring{R}\backslash\{0\}$, we say that $f$ reduces to $h$ modulo $G$ denoted
	\begin{align*}
		f\xrightarrow{G}h
	\end{align*}
	if there exists a finite sequence $g_1,\cdots, g_i\in G$ and elements $h_1,\cdots h_{i-1}\in \ring{R}$ such that:
	\begin{align*}
		f\xrightarrow{\{g_1\}}h_1\xrightarrow{\{g_2\}}h_2\xrightarrow{}\cdots
		\xrightarrow{} 
		h_{i-1}\xrightarrow{\{g_i\}} h.
	\end{align*}

	Let $f,g\in \ring{R}\backslash \{0\}$, and let $L$ be the least common multiple of their highest \emph{monomials} (i.e., we ignore the coefficients of these monomials). We define:
	\begin{align}
		S(f,g)&=\dfrac{L}{\lt(f)}f-\dfrac{L}{\lt(g)}g.
	\end{align}
	
	\begin{thm}[Buchberger's Criterion {\cite[Ch.\ 2 \S 6, Thm.\ 6]{CoxLitOSh}}]
		Let $\ring{I}$ be a homogeneous ideal of $\ring{R}$ generated by  a
		set $G\subseteq \ring{I}\backslash\{0\}$.
		Then, $G$ is a Gr\"obner basis for $\ring{I}$ if and only if for all $f,g\in G$, $S(f,g)\xrightarrow{G}0$.
	\end{thm}
	
	The main engine for this paper is the following lemma.
	
	\begin{lem}
	\label{lem:main}
	Let $\ring{I}$ be a weight-homogeneous ideal, let $G\subseteq \ring{I}\backslash \{0\}$ be a generating set for $\ring{I}$, each element of which is weight-homogeneous. Let 
	$$\ring{K}=\langle\lt(g)\,|\,g\in G\rangle.$$ 
	If we have the coefficient-wise inequality 
	\begin{align*}
		H_{\ring{K}}(q)
		&\leq H_{\ring{I}}(q),
	\end{align*}
	then,
	\begin{enumerate}
		\item we have $\chi(q)=H_{\ring{I}}(q)$,
		\item $G$ is a Gr\"obner basis for $\ring{I}$. 
	\end{enumerate}
	\end{lem}
	\begin{proof}
	Clearly, $\ring{K}\subseteq \LT(\ring{I})$ and thus, coefficient-wise, we have:
	$H_{\ring{I}}(q)=H_{\LT(\ring{I})}(q)\leq H_{\ring{K}}(q).$
	Combining with $H_{\ring{K}}(q)\leq H_{\ring{I}}(q)$, we get the equality $H_{\ring{K}}(q)=H_{\ring{I}}(q)$.

	Therefore, $\LT(\ring{I})=\ring{K}$ and $G$ is a Gr\"obner basis for $\ring{I}$, see \cite[Ch.\ 10, \S 2, Prop.\ 1]{CoxLitOSh}.
	\end{proof}

	\section{Setup}
	\label{sec:setup}
	In this section, we now setup the conventions for the calculations in $J_{\infty}R_{\VOA{L}_1(\la{sl}_n)}$, where we fix an $n\geq 2$.
	
	Recall that $E(i,j)$ is the $n\times n$ elementary matrix that has a $1$ in row $i$ and column $j$ and zeros everywhere else.
	Now define:
	\begin{align}
		X(i,j)=
		\begin{cases}
			E(i,j) & 1\leq i\neq j\leq n\\
			E(i,i) - E(i+1,i+1) & 1\leq i=j\leq n-1.
		\end{cases}
	\end{align}
	It will be convenient to give a name to the set of colours:
	\begin{align}
		\sX_n=\{X(i,j)\,|\,1\leq i,j\leq n, (i,j)\neq (n,n)\}.
		\label{eqn:Xsymb}
	\end{align}
	
	The Lie algebra $\la{g}=\la{sl}_n$ is spanned by $X(i,j)\in\sX_n$. We have:
	\begin{align*}
		\la{g}=\la{n}_-\oplus \la{h} \oplus \la{n}_+
	\end{align*}
	where
	\begin{align*}
		\la{n}_+&=\mathrm{Span}_{\CC}\{ X(i,j)\,|\, 1\leq i<j\leq n\},\\
		\la{n}_-&=\mathrm{Span}_{\CC}\{ X(i,j)\,|\, 1\leq j<i\leq n\},\\
		\la{h}&=\mathrm{Span}_{\CC}\{ X(i,i)\,|\, 1\leq i\leq n-1\}.
	\end{align*}
	
	We will work within the commutative ring
	\begin{align}
		\ring{R} = \CC[X(i,j)_{k}\,|\, X(i,j)\in\sX_n, k\geq 1].
		\label{eqn:ringR}
	\end{align}
	Note that even though the elements $E(i,i)_k$ are not officially elements of $\ring{R}$, it may be sometimes beneficial to use expressions such as:
	\begin{align*}
		E(i,i)_k -E(j,j)_k,
	\end{align*}
	with $j>i$.
	We shall naturally interpret this as the element
	\begin{align*}
		X(i,i)_k + \cdots + X(j-1,j-1)_{k} \in \ring{R}.
	\end{align*}
	
	\begin{defi}
	We define a total order $\succ$ on colours in $\sX_n$.
	This order is uniquely defined by:
	\begin{enumerate} 
	\item For $X$'s belonging to the respective spaces:
	\begin{align}
		\la{n}_+ \succ \la{h} \succ \la{n}_-,
	\end{align}
	\item For $1\leq i<j\leq n$ and $1\leq i'<j'\leq n$
	\begin{align}
		\label{eqn:Xord}
		X_{i,j} \succ X_{i',j'} \iff X_{j,i} \prec X_{j',i'},
	\end{align}
	\item On $\la{n}_+$, that is, for $1\leq i<j \leq n$ and $1\leq i'<j' \leq n$
	\begin{align}
		X_{i,j} \succ  X_{i',j'} \iff i<i'\,\, \mathrm{or}\,\, i=i', j>j',
	\end{align}
	\item On $\la{h}$, that is, for $1\leq i,j\leq n-1$,
	\begin{align}
		X_{i,i}\succ X_{j,j} \iff i<j.
	\end{align}
	\end{enumerate}
	\end{defi}
	
	In words, the largest $X$ is $X(1,n)$, in the top-right corner. Within $\la{n}_+$, to get successively lower $X$'s, we start with the top row, scan it from right to left and then once the top row is finished, we move successively to the bottom rows, again scanning each of them from right to left. Within $\la{h}_+$, we just move along the diagonal from top to bottom to get successively lower elements. The order on $\la{n}_-$ is determined by transposing and reversing the one on $\la{n}_+$.
	
	As an example for $\la{sl}_4$, we have the following order, with $1$ being the highest.
	\begin{align*}
		\begin{matrix}
			X(1,1) & X(1,2) &X(1,3) &X(1,4)\\
			X(2,1) & X(2,2) &X(2,3) &X(2,4)\\
			X(3,1) & X(3,2) &X(3,3) &X(3,4)\\
			X(4,1) & X(4,2) &X(4,3) &
		\end{matrix}
		\quad\quad 
		\begin{matrix}
			7 & 3 &2 &1\\
			13 & 8 &5 &4\\
			14 & 11 &9 &6\\
			15 & 12 &10 &
		\end{matrix}
	\end{align*}
	\begin{rem}
		\label{rem:toprightimpliesgreater}
		If $x,y\in\sX_n$ and $x$ lies strictly to the top-right of $y$, or to the right of $y$ in the same row, or to the top of $y$ in the same column, then $x\succ y$.
	\end{rem}
	
	\begin{defi}	
	\label{def:ourgrevlex}
	We define the following total order on $\ZZ_{>0}\times \sX_n$:
	\begin{align}
		X(i,j)_{k} \succ X(i',j')_{k'} \iff k<k' \,\,\mathrm{or}\,\, k=k', X_{i,j}\succ X_{i',j'}.
	\end{align}
	\end{defi}
	Letting
	\begin{align*}
		\wt(X(i,j)_k)=k,
	\end{align*}	
	completes the definition of the grevlex order on $\ring{R}$.
	
	On $\ring{R}$, we have the action of $\la{g}$ by derivations. That is, we have:
	$$
	\mathrm{ad}(X(i,j))(X(i',j')_{k})=
	[X(i,j), X(i',j')_{k}]=[X(i,j), X(i',j')]_{k},$$
	where the bracket on the right-hand side is calculated in $\la{g}$.
	This action is extended by linearity and the Leibniz rule to act on
	$\ring{R}$.
	Next, we have a derivative $\partial$  which acts by:
	\begin{align}
		\partial X(i,j)_{k} = kX(i,j)_{k+1}.
		\label{eqn:der}
	\end{align}
	
	\begin{defi} 
	\label{def:TI}
	We let:
	\begin{align}
		\ring{T}&=\mathrm{ad}(\la{U}(\la{g}))\cdot(X(1,n)_1^2)
		\label{eqn:Tspace},\\
		\ring{I}&=\langle \partial^k\ring{T}\,|\, k\geq 0\rangle.
		\label{eqn:idealI}
	\end{align}
	\end{defi}
	We thus see that Zhu's $C_2$ algebra for $\VOA{L}_1(\la{sl}_n)$ is:
	\begin{align*}
		R_{\VOA{L}_1(\la{sl}_n)}=\dfrac{\CC[X(i,j)_1\,|\, X(i,j)\in\sX_n]}{\langle \ring{T} \rangle},
	\end{align*}
	and that our object of interest is the arc algebra:
	\begin{align}
		J_{\infty}R_{\VOA{L}_1(\la{sl}_n)}=
		\dfrac{\CC[\partial^kX(i,j)_1\,|\, k\geq 0, X(i,j)\in\sX_n]}{\langle \partial^k\ring{T}\,|\,k\geq 0 \rangle}\cong \dfrac{\ring{R}}{\ring{I}}.
		\label{eqn:JRsln1}
	\end{align}

	A few easy observations will help in the subsequent calculations.	

	First, for a formal variable $z$ and $X\in\sX_n$, consider	
	\begin{align*}
		X(z)=\sum_{n\geq 0}X_{n+1}z^n.
	\end{align*}
	Then by definition \eqref{eqn:der} of $\partial$ 
	\begin{align*}
		\frac{d}{dz}X(z)=\partial X(z),
	\end{align*}
	and by an easy extension, for $X^{(1)}, X^{(2)},\cdots, X^{(i)}\in \sX_n$,
	\begin{align*}
		\frac{d}{dz}\left(X^{(1)}(z)\cdots X^{(i)}(z)\right)
		=\partial\left(X^{(1)}(z)\cdots X^{(i)}(z)\right).
	\end{align*}
	Iterating this relation and then extracting coefficients, we get the next lemma.
	\begin{lem}\label{lem:derprod}
	For $X^{(1)}, X^{(2)},\cdots, X^{(i)}\in \sX_n$ and $k\in\ZZ_{\geq 0}$ we have
	\begin{align}
		\partial^{k}(X^{(1)}_1X^{(2)}_1\cdots X^{(i)}_1 ) = k! \sum_{\substack{(n_1,n_2,\cdots,n_i)\in \ZZ_{\geq 0}^i \\ n_1+n_2+\cdots+n_i=k } } X^{(1)}_{1+n_1}X^{(2)}_{1+n_2}\cdots X^{(i)}_{1+n_i}.
	\end{align}		
	\end{lem}

	\begin{nota}
		While dealing with derivatives, it will convenient to use two notations.
		First, we let
		$$\partial^{[k]}=\dfrac{\partial^k }{k!}.$$
		Next,
		we say $a\doteq b$ if $a=\alpha b$ for some $\alpha\in \CC^{\ast}$.
	\end{nota}

	Next lemma will be used repeatedly to determine various leading terms.
	\begin{lem}
		\label{lem:quadder}
		Let  $a,b\in\CC^{\star}$, $k\in\ZZ_{\geq 1}$, and $X^{(1)}\succeq X^{(2)}\succeq X^{(3)}\succeq X^{(4)}\in \sX_n$ such that $X^{(1)}_1X^{(4)}_1\neq X^{(2)}_1X^{(3)}_1$ (i.e., either $X^{(1)}\succ X^{(2)}$ or $X^{(3)}\succ X^{(4)}$).
		\begin{enumerate}
			\item If $X^{(1)}\succ X^{(2)}$ and  $X^{(3)}\succ X^{(4)}$ then:
			\begin{align*}
			\lt(\partial^{2k-1}(aX^{(1)}_1X^{(4)}_1+bX^{(2)}_1X^{(3)}_1))
				&\doteq  X^{(4)}_kX^{(1)}_{k+1},\\
				\lt(\partial^{2k-2}(aX^{(1)}_1X^{(4)}_1+bX^{(2)}_1X^{(3)}_1))
				&\doteq  X^{(2)}_kX^{(3)}_{k}.
			\end{align*}
			\item If $X^{(1)}\succ X^{(2)}$ and  $X^{(3)}=X^{(4)}$ then: 
			\begin{align*}
				\lt(\partial^{2k-1}(aX^{(1)}_1X^{(3)}_1+bX^{(2)}_1X^{(3)}_1))
				&\doteq  X^{(3)}_kX^{(1)}_{k+1},\\
				\lt(\partial^{2k-2}(aX^{(1)}_1X^{(3)}_1+bX^{(2)}_1X^{(3)}_1))
				&\doteq  X^{(1)}_kX^{(3)}_{k}.
			\end{align*}
			\item If $X^{(1)}= X^{(2)}$ and  $X^{(3)}\succ X^{(4)}$ then: 
			\begin{align*}
				\lt(\partial^{2k-1}(aX^{(1)}_1X^{(3)}_1+bX^{(1)}_1X^{(4)}_1))
				&\doteq  X^{(3)}_kX^{(1)}_{k+1},\\
				\lt(\partial^{2k-2}(aX^{(1)}_1X^{(3)}_1+bX^{(1)}_1X^{(4)}_1))
				&\doteq  X^{(1)}_kX^{(3)}_{k}.
			\end{align*}
		\end{enumerate}
	\end{lem}

	\section{Elements of \texorpdfstring{$\ring{T}$}{T}}
	\label{sec:quad}
	
	The aim of this section is to produce a set of elements of $\ring{T}$ such that these elements, their derivatives, and the $S$-polynomials among them eventually form our Gr\"obner basis for $\ring{I}$.
	
	To begin, we let 
	\begin{align}
		G_1 = \{\}
	\end{align}		
	and as we go along, we shall put certain elements of $\ring{T}$ and their derivatives in $G_1$. The set $G_1$ will be the starting point of Buchberger's algorithm.

	By a slight abuse of notation, we let 
	\begin{align*}
		\la{h}=\{(i,i)\,|\, 1\leq i\leq n-1\},\quad \la{h}^e=\{(i,i)\,|\, 1\leq i\leq n\}.
	\end{align*}

	Throughout this section, we will consider four indices $i,j,i',j'$,  $1\leq i,j,i',j'\leq n$, and produce elements of $\ring{T}$ corresponding to the box with corners $(i,j)$, $(i',j')$, $(i,j')$, $(i',j)$. These elements will have the property that, the leading terms of their odd order derivatives (especially, the first derivatives which yield weight $3$ elements) will be quadratics whose factors lie on the top-right and bottom-left corners, see Figure \ref{fig:boxes}.
	\begin{figure}
		\centering
		\subcaptionbox{Leading terms in $\partial^{2k-1} \ring{T}$ for the allowed boxes\label{fig:boxes}}[0.22\linewidth]{
			\begin{tikzpicture}[scale=0.8]
				\draw[dotted] (0.5,1) -- (2.5,1);
				\draw[dotted] (0.5,2) -- (2.5,2);
				\draw[dotted] (1,0.5) -- (1,2.5);
				\draw[dotted] (2,0.5) -- (2,2.5);
				\node at (0.2,1) {$\scriptstyle{i}$};
				\node at (0.2,2) {$\scriptstyle{i'}$};
				\node at (1,2.8) {$\scriptstyle{j'}$};
				\node at (2,2.8) {$\scriptstyle{j}$};
				\node[fill=white, inner sep=0.3pt] at (2,2) {$\scriptstyle{\bullet}$};
				\node[fill=white, inner sep=0.3pt] at (1,1) {$\scriptstyle{\bullet}$};
				\node at (2.4,1.8) {$\scriptstyle{k+1}$};
				\node at (1.2,0.8) {$\scriptstyle{k}$};
			\end{tikzpicture}
		}	
		\subcaptionbox{Not allowed\label{fig:notalloweddeg}}[0.22\linewidth]{
			\begin{tikzpicture}[scale=1]
				\draw[dotted] (0.5,1)--(1.5,1);
				\draw[dotted] (1,0.5)--(1,1.5);
				\draw[dotted] (0.5,1.5) -- (1.5,0.5);
				\node at (1.7,0.3) {$\scriptstyle{\la{h}^e}$};
				\node[fill=white, inner sep=0.3pt] at (1,1) {$\scriptstyle{\bullet}$};
				\node[align=center] at (0,1) {$\scriptstyle{i'=i}$\\$\scriptstyle{=j'=j}$};
				\node[align=center] at (1,2) {$\scriptstyle{j'=j}$\\$\scriptstyle{=i'=i}$};
				\end{tikzpicture}
		}	
		\subcaptionbox{Not allowed\label{fig:notallowedvert}}[0.22\linewidth]{
			\begin{tikzpicture}[scale=1]
				\draw[dotted] (0.5,1)--(1.5,1);
				\draw[dotted] (0.5,2)--(1.5,2);
				\draw[dotted] (1,0.5)--(1,2.5);
				\draw[dotted] (0.5,1.5) -- (1.5,0.5);
				\node at (1.7,0.3) {$\scriptstyle{\la{h}^e}$};
				\node[fill=white, inner sep=0.3pt] at (1,1) {$\scriptstyle{\bullet}$};
				\node[fill=white, inner sep=0.3pt] at (1,2) {$\scriptstyle{\bullet}$};
				\node[align=center] at (1,2.7) {$\scriptstyle{j=j'=i}$};
				\node[align=center] at (0,2) {$\scriptstyle{i'}$};
				\node[align=center,rotate=-90] at (0,1.5) {$\scriptstyle{<}$};
				\node[align=center] at (0,1) {$\scriptstyle{i=j=j'}$};
			\end{tikzpicture}
		}
		\subcaptionbox{Not allowed\label{fig:notallowedhoriz}}[0.22\linewidth]{
			\begin{tikzpicture}[scale=1]
				\draw[dotted] (0.5,1)--(2.5,1);
				\draw[dotted] (2,0.5)--(2,1.5);
				\draw[dotted] (1,0.5)--(1,1.5);
				\draw[dotted] (1.5,1.5) -- (2.5,0.5);
				\node at (2.7,0.3) {$\scriptstyle{\la{h}^e}$};
				\node[fill=white, inner sep=0.3pt] at (1,1) {$\scriptstyle{\bullet}$};
				\node[fill=white, inner sep=0.3pt] at (2,1) {$\scriptstyle{\bullet}$};
				\node[align=center] at (2,1.8) {$\scriptstyle{j=i=i'}$};
				\node[align=center] at (1.3,1.8) {$\scriptstyle{<}$};
				\node[align=center] at (1,1.8) {$\scriptstyle{j'}$};
				\node[align=center] at (0,1) {$\scriptstyle{i'=i=j}$};
			\end{tikzpicture}
		}	
	\caption{Boxes and leading terms of odd order derivatives}
	\label{fig:iji'j'}
	\end{figure}
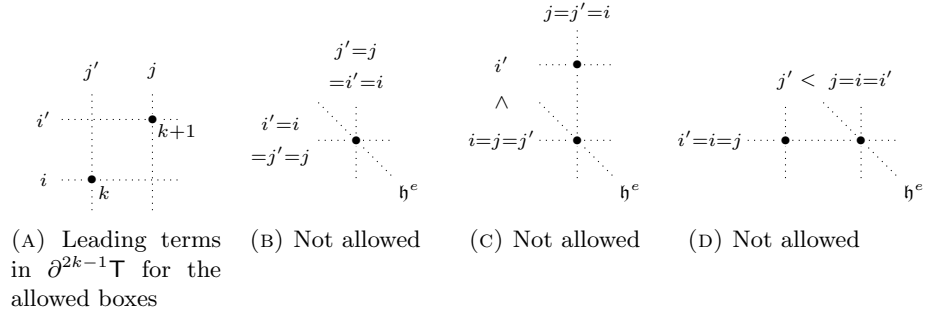
	
	Several  such boxes will \emph{not} be considered.
	Specifically, we will \emph{not} consider indices such that:
	
	\begin{enumerate}
		\item The box is completely degenerate  and lies on $\la{h}^e$ (see Figure \ref{fig:notalloweddeg}), i.e., $$1\leq i=i'=j=j'\leq n.$$
		\item The box is semi-degenerate, i.e., has only two distinct points, and such that these points are vertical with the bottom one on $\la{h}^e$ (see Figure \ref{fig:notallowedvert}), i.e.,
		$$1\leq i'<i=j=j'\leq n.$$
		\item The transpose of the previous kind of box. That is, the box is semi-degenerate (has only two distinct points) and such that these points are  horizontal with the right one on $\la{h}^e$ (see Figure \ref{fig:notallowedhoriz}), i.e., 
		$$1\leq j'<j=i=i'\leq n.$$
	\end{enumerate}

	It will be beneficial to denote the singular vector by $v$:
	\begin{align}
		v=X(1,n)_{1}^2\in \ring{T} \subset \ring{I}
		\label{eqn:sing}
	\end{align}
	Note that:
	\begin{align*}
		v^{\tr}=X(n,1)_{1}^2\in \ring{T}
	\end{align*}		
	as well, because,
	\begin{align*}
	X(1,n)_{1}^2&\xrightarrow{\mathrm{ad}(X(n,1))}
	2( E(n,n)_{1}-E(1,1)_{1})X(1,n)_{1}\\
	&\xrightarrow{\mathrm{ad}(X(n,1))}
	-4X(n,1)_{1}X(1,n)_{1}+ 2(E(n,n)_{1}-E(1,1)_{1})^2\\
	&\xrightarrow{\mathrm{ad}(X(n,1))}
	-12X(n,1)_{1}(E(n,n)_{1}-E(1,1)_{1})\\
	&\xrightarrow{\mathrm{ad}(X(n,1))}
	24X(n,1)_{1}^2=24v^{\tr}\in \ring{T}.
	\end{align*}
	
	\begin{lem} \label{lem:transpose} If $x\in \ring{T}$, $x^\tr\in \ring{T}$.
	\end{lem}
	\begin{proof}
		Clearly, $v\in \ring{T}$, and $v^{\tr}\in \ring{T}$. Then,
		for $a\in \la{g}$,
		$[a,v]\in \ring{T}$, and $[a, v]^\tr=-[a^{\tr},v^{\tr}]\in \ring{T}$. Now, by induction on $j$, it follows that for all $a_1,\cdots, a_j\in \la{g}$,
		$[a_1,\cdots,[a_j,v]\cdots]^{\tr}\in \ring{T}$. Since $\ring{T}$ is spanned by such elements, it is stable under transpose.
	\end{proof}
	
	\subsection{Completely degenerate box} 
	\label{sec:compldeg}
	We start with the completely degenerate box that does not lie on $\la{h}^e$, i.e., $i=i'$, $j=j'$  but $i\neq j$.
	
	If $j=n$, $i<n$, we may assume that $i> 1$, since $i=1$ corresponds to $v\in \ring{T}$. Then, we see:
	\begin{align*}
	 v\xrightarrow{\mathrm{ad}(X(i,1))}2X(i,n)_{1}X(1,n)_{1}\xrightarrow{\mathrm{ad}(X(i,1))}
	 2X(i,n)_{1}^2
	 \in \ring{T}.
	\end{align*}
	Thus, to the box with corners $(i,n),(i,n),(i,n), (i,n)$, we shall associate the element
	\begin{align*}
		x=X(i,n)^2_{1}\in \ring{T}.
	\end{align*}
	It is then clear that for $k\in\ZZ_{\geq 1}$:
	\begin{align}
		\lt(\partial^{2k-1} x) &\doteq X(i,n)_{k}X(i,n)_{k+1}.\\
		\lt(\partial^{2k-2} x) &\doteq X(i,n)_{k}X(i,n)_{k}.
	\end{align}
	
	Next, if $i,j<n$ and $i\neq j$, we see:
	\begin{align*}
		X(i,n)^2_{1}\xrightarrow{\mathrm{ad}(X(n,j))}
		-2X(i,n)_{1}X(i,j)_{1}
		\xrightarrow{\mathrm{ad}(X(n,j))}
		2X(i,j)^2_{1}\in \ring{T}.
	\end{align*}
	Thus, to the box $(i,j),(i,j),(i,j), (i,j)$, we shall associate the element
	\begin{align}
		x=X(i,j)_{1}^2\in \ring{T}. \label{eqn:compdegnotonh}
	\end{align}
	It is then clear that for $k\in\ZZ_{\geq 1}$:
	\begin{align}
		\lt(\partial^{2k-1} x) &\doteq X(i,j)_{k}X(i,j)_{k+1},\\
		\lt(\partial^{2k-2} x) &\doteq X(i,j)_{k}X(i,j)_{k}.
	\end{align}

	This takes care of indices $i\neq j$, $1\leq i<n$, $1\leq j\leq n$.
	For the remaining, 
	we simply see take transpose of $X(i,n)^2_{1}\in \ring{T}$, ($1\leq i<n$) to obtain 
	\begin{align*}
		x=X(n,i)^2_{1}\in \ring{T},
	\end{align*}
	and for $k\in\ZZ_{\geq 1}$:
	\begin{align*}
		\lt(\partial^{2k-1} x) &\doteq X(n,i)_{k}X(n,i)_{k+1},\\
		\lt(\partial^{2k-2} x) &\doteq X(n,i)_{k}X(n,i)_{k}.
	\end{align*}
	
	In conclusion, for all $1\leq i\neq j\leq n$, we have produced:
	\begin{align}
		X(i,j)_1^2&\in \ring{T} \label{eqn:deg}.
	\end{align}
	Now we enlarge the set $G_1$ by putting in it these elements and their derivatives.
	
	\begin{nota}
		Going forward, we will denote the designation  $\mathrm{ad}(X(i,j))$ simply by $X(i,j)$. 
	\end{nota}
	
	\subsection{Semi-degenerate horizontal box}
	Next, we consider the case when the box is semi-degenerate horizontal. 
		
	We start with the case that no endpoint is in $\la{h}^e$. So, we assume that $i=i'$, $j'<j$, but $i\neq j$ and $i\neq j'$.
	For this, we observe (see Figure \ref{fig:semideghoriz})
	\begin{align}
		X(i,j)^2\in \ring{T} \xrightarrow{-\frac{1}{2}X(j,j')} x=\ord{X(i,j)_{1}}{1}\ord{X(i,j')_{1}}{2}\in \ring{T}.
		\label{eqn:semideghoriz}
	\end{align}
	Here and onwards, we will denote the order of $X(i,j)_k$ factors with numerical marks under them, with 
	\tikz{\node[ellipse,draw,inner sep=1pt,minimum width=1em, minimum height=1em]  {$\scriptstyle{1}$};} denoting the highest factor.
		
	We clearly have for $k\in\ZZ_{\geq 1}$:
	\begin{align*}
		\lt(\partial^{2k-1}x)&\doteq X(i,j')_{k}X(i,j)_{k+1},\\
		\lt(\partial^{2k-2}x)&\doteq X(i,j)_{k}X(i,j')_{k}.
	\end{align*}

	If however, the left end point of the box is on $\la{h}$, i.e., if $i=i'=j'<j$, we proceed as follows. 

	We start with $X(i,j)_{1}^2\in \ring{T}$ and use:
	\begin{align}
		&X(i,j)_{1}^2\xrightarrow{-\frac{1}{2}X(j,i)}X(i,j)_{1}(E(i,i)-E(j,j))_{1}\notag\\
		&=\ord{X(i,j)_{1}}{1}\ord{X(i,i)_{1}}{2} + \cdots +\ord{X(i,j)_{1}}{1}\ord{X(j-1,j-1)_{1}}{j-i+1} =x\in \ring{T}.
		\label{eqn:semideghorizh}
	\end{align}
	For the locations of the elements, refer to Figure \ref{fig:semideghoriz}.
	Comparing leading terms arising from various quadratics using Lemma \ref{lem:quadder} we see:
	\begin{align*}
		\lt(\partial^{2k-1}x)&\doteq X(i,i)_{k}X(i,j)_{k+1},\\
		\lt(\partial^{2k-2}x)&\doteq X(i,j)_{k}X(i,i)_{k}.
	\end{align*}
	
	We enlarge the set $G_1$ by putting in it elements produced in \eqref{eqn:semideghoriz}, \eqref{eqn:semideghorizh}, and their derivatives of all orders.
	
	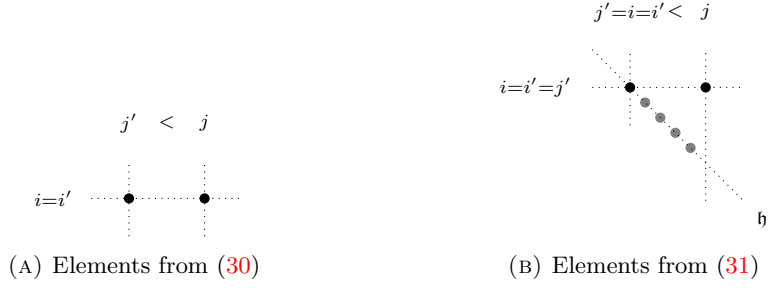
\begin{figure}
		\centering
		\captionsetup{width=\linewidth}
		\subcaptionbox{Elements from \eqref{eqn:semideghoriz}}[0.43\linewidth]
		{	\begin{tikzpicture}[scale=1]
				\node (i) at (0,1) {$\scriptstyle{i=i'}$};
				\node (j) at (1,2) {$\scriptstyle{j'}$};
				\node (jp) at (1.5,2) {$\scriptstyle{<}$};
				\node (jp) at (2,2) {$\scriptstyle{j}$};
				\draw[dotted] (0.5,1) -- (2.5,1);
				\draw[dotted] (1,0.5) -- (1,1.5);
				\draw[dotted] (2,0.5) -- (2,1.5);
				\draw[black,fill=black] (1,1) circle (0.6mm);
				\draw[black,fill=black] (2,1) circle (0.6mm);
		\end{tikzpicture}
		}
		\qquad\quad
		\subcaptionbox{Elements from \eqref{eqn:semideghorizh}}[0.43\linewidth]
		{
			\begin{tikzpicture}[scale=1]
				\node at (2.75,0.25) {$\scriptstyle{\la{h}}$};
				\draw[dotted] (0.5,2) -- (2.5,2);
				\draw[dotted] (1,1.5) -- (1,2.5);
				\draw[dotted] (2,0.5) -- (2,2.5);
				\node at (-0.25,2) {$\scriptstyle{i=i'=j'}$};
				\node at (1,3) {$\scriptstyle{j'=i=i'}$};
				\node at (1.6,3) {$\scriptstyle{<}$};
				\node at (2,3) {$\scriptstyle{j}$};
				\draw[black,fill=black] (2,2) circle (0.6mm);
				\draw[black,fill=black] (1,2) circle (0.6mm);
				\draw[gray,fill=gray] (1.2,1.8) circle (0.6mm);
				\draw[gray,fill=gray] (1.4,1.6) circle (0.6mm);
				\draw[gray,fill=gray] (1.6,1.4) circle (0.6mm);
				\draw[gray,fill=gray] (1.8,1.2) circle (0.6mm);
				\draw[dotted] (0.5,2.5) -- (2.5,0.5);
				
			\end{tikzpicture}
		}
		\caption{Semi-degenerate horizontal box, right end-point not in $\la{h}^e$}
		\label{fig:semideghoriz}
	\end{figure}
	
	\subsection{Semi-degenerate vertical box}
	These elements are obtained by transposing the elements obtained above, see Lemma \ref{lem:transpose}.
	Refer to Figure \ref{fig:semidegvert} for reference.
	If $j=j'$, $i'<i$, but $i\neq j$, $i'\neq j'$, we have:
	\begin{align}
		x=\ord{X(i',j)_{1}}{1}\ord{X(i,j)_{1}}{2}\in \ring{T},
		\label{eqn:semidegvert}
	\end{align}
	with
	\begin{align*}
		\lt(\partial^{2k-1}x)&\doteq X(i,j)_{k}X(i',j)_{k+1},\\
		\lt(\partial^{2k-2}x)&\doteq X(i',j)_{k}X(i,j)_{k}.
	\end{align*}
	If the top endpoint is in $\la{h}$ i.e., $i'=j=j'<i$, 
	\begin{align}
		x=\ord{X(i',i')_{1}}{1} \ord{X(i,j)_{1}}{i-i'+1}+ \cdots +\ord{X(i-1,i-1)_{1}}{i-i'}\ord{X(i,j)_{1}}{i-i'+1} \in \ring{T}
		\label{eqn:semidegverth},
	\end{align}
	with 
	\begin{align*}
		\lt(\partial^{2k-1}x)&\doteq  X(i,j)_kX(i',i')_{k+1},\\
		\lt(\partial^{2k-2}x)&\doteq  X(i',i')_{k}X(i,j)_k.
	\end{align*}
	We enlarge the set $G_1$ by putting in it elements \eqref{eqn:semidegvert}, \eqref{eqn:semidegverth}, and their derivatives of all orders.
	\begin{figure}
		\centering
		\captionsetup{width=\linewidth}
		\subcaptionbox[0.43\linewidth]{Elements from \eqref{eqn:semidegvert}}
		{
			\begin{tikzpicture}[scale=1]
				\node at (-0.5,1) {};
				\node at (2.5,1) {};
				\node  at (1,3) {$\scriptstyle{j'=j}$};
				\node  at (0,1) {$\scriptstyle{i}$};
				\node[rotate=-90]  at (0,1.5) {$\scriptstyle{<}$};
				\node  at (0,2) {$\scriptstyle{i'}$};
				\draw[dotted] (1,0.5) -- (1,2.5);
				\draw[dotted] (0.5,1) -- (1.5,1);
				\draw[dotted] (0.5,2) -- (1.5,2);
				\draw[black,fill=black] (1,1) circle (0.6mm);
				\draw[black,fill=black] (1,2) circle (0.6mm);
			\end{tikzpicture}
		}
		\qquad\quad
		\subcaptionbox[0.43\linewidth]{Elements from \eqref{eqn:semidegverth}}
		{
			\begin{tikzpicture}[scale=1]
				\node at (2.75,0.25) {$\scriptstyle{\la{h}}$};
				\draw[dotted] (0.5,1) -- (2.5,1);
				\draw[dotted] (1,0.5) -- (1,2.5);
				\draw[dotted] (0.5,2) -- (1.5,2);
				\node at (0,2) {$\scriptstyle{i'=j'=j}$};
				\node[rotate=-90] at (0,1.5) {$\scriptstyle{<}$};
				\node at (1,3) {$\scriptstyle{j'=j=i'}$};
				\node at (0,1) {$\scriptstyle{i}$};
				\draw[black,fill=black] (1,1) circle (0.6mm);
				\draw[gray,fill=black] (1,2) circle (0.6mm);
				\draw[gray,fill=gray] (1.2,1.8) circle (0.6mm);
				\draw[gray,fill=gray] (1.4,1.6) circle (0.6mm);
				\draw[gray,fill=gray] (1.6,1.4) circle (0.6mm);
				\draw[gray,fill=gray] (1.8,1.2) circle (0.6mm);
				\draw[dotted] (0.5,2.5) -- (2.5,0.5);
			\end{tikzpicture}
		}
		\caption{Semi-degenerate vertical box, bottom end-point not in $\la{h}^e$}
		\label{fig:semidegvert}
	\end{figure}
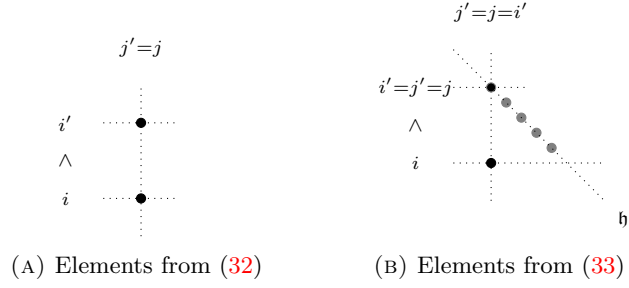

	Next, we will consider non-degenerate cases. So, we assume that $i'<i$, $j'<j$.
	
	\subsection{Non-degenerate and no location in \texorpdfstring{$\la{h}^e$}{h e}}
    Suppose that no location among $(i,j), (i',j'), (i,j'), (i',j)$ is in $\la{h}^e$. Thus, $i\neq j$, $i'\neq  j'$.
    Here, we observe:
    \begin{align}
    	X(i,j)_{1}X(i,j')_{1}\in \ring{T}\xrightarrow{X(i',i)}
    	x=\ord{X(i',j)_{1}}{1}\ord{X(i,j')_{1}}{4}+X(i,j)_{1}X(i',j')_{1}\in \ring{T},
    	\label{eqn:nondegnohe}
    \end{align}
    see Figure \ref{fig:nondegnoh}.
    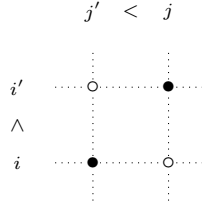
\begin{figure}
    	\begin{tikzpicture}[scale=1]
    		\draw[dotted] (0.5,1) -- (2.5,1);
    		\draw[dotted] (0.5,2) -- (2.5,2);
    		\draw[dotted] (1,0.5) -- (1,2.5);
    		\draw[dotted] (2,0.5) -- (2,2.5);
    		\node at (0,1) {$\scriptstyle{i}$};
    		\node[rotate=-90] at (0,1.5) {$\scriptstyle{<}$};
    		\node at (0,2) {$\scriptstyle{i'}$};
    		\node at (1,3) {$\scriptstyle{j'}$};
    		\node at (1.5,3) {$\scriptstyle{<}$};
    		\node at (2,3) {$\scriptstyle{j}$};
    		\node[fill=white,inner sep=0.3pt] at (2,2) {${\bullet}$};
    		\node[fill=white,inner sep=0.3pt] at (1,1) {${\bullet}$};
    		\node[fill=white,inner sep=0.3pt] at (2,1) {${\circ}$};
    		\node[fill=white,inner sep=0.3pt] at (1,2) {${\circ}$};
    	\end{tikzpicture}
    	\caption{Non-degenerate box, no location in $\la{h}^e$}
    	\label{fig:nondegnoh}
	\end{figure}    	
    Here, the exact ordering among $X(i,j)_{1}$ and $X(i',j')$ depends on how they belong to $\la{n}_+$ or $\la{n}_-$, however, regardless, the highest and the lowest $X$ terms remain as indicated. Thus, applying Lemma \ref{lem:quadder}, we see
    \begin{align*}
    	\lt(\partial^{2k-1} x) &\doteq X(i,j')_{k}X(i',j)_{k+1},\\
    	\lt(\partial^{2k-2} x) &\doteq X(i,j)_{k}X(i',j')_{k}.
    \end{align*}
    We include the elements \eqref{eqn:nondegnohe} and their derivatives in $G_1$.

    \subsection{Non-degenerate and top-left in \texorpdfstring{$\la{h}$}{h}}
    In addition to $i'<i$ and $j'<j$, we assume that $i'=j'$, but $i\neq j$, i.e., only the top-left location is in $\la{h}$.
	Locations of various elements that arise in this section are collected in Figure \ref{fig:nondegtlinh}.

	First, suppose that $i=i'+1$. Since $i'=j'<j$, $i\neq j$, this implies that $i'+1<j$.
	Starting from the element of the type \eqref{eqn:semideghorizh} (with  appropriate indices) we obtain:
	\begin{align}
		X(i',&i')_{1}X(i',i'+1)_{1} = (E(i',i')_{1}-E(i'+1,i'+1)_{1})X(i',i'+1)_{1}\notag\\ &\xrightarrow{-X(i'+1,j)}
		-X(i'+1,j)_{1}X(i',i'+1)_{1}+X(i',i')_{1}X(i',j)_{1}\notag\\
		&\xrightarrow{\frac{1}{2}X(i'+1,i')}
		x=\ord{X(i'+1,j)_{1}}{2}\ord{X(i',i')_{1}}{3}
		+\ord{X(i'+1,i')_{1}}{4}\ord{X(i',j)_{1}}{1}
		\in \ring{T}.
		\label{eqn:nondegtlinhshort}
	\end{align}
	Thus, 
	\begin{align*}
		\lt(\partial^{2k-1} x) &\doteq X(i'+1,i')_{k}X(i',j)_{k+1},\\
		\lt(\partial^{2k-2} x) &\doteq X(i'+1,j)_{k}X(i',i')_{k}.
	\end{align*}
	
	If $j=j'+1=i'+1$ then we may transpose the element just obtained, but we restore the indices as in Figure \ref{fig:iji'j'}, i.e., after transposing, we use $i\leftrightarrow j$, $i\leftrightarrow j'$.
	Thus,
	\begin{align}
		x=\ord{X(i,j'+1)_{1}}{3}\ord{X(j',j')_{1}}{2}
		+\ord{X(j',j'+1)_{1}}{1}\ord{X(i,j')_{1}}{4}
		\in \ring{T},
		\label{eqn:nondegtlinhthin}
	\end{align}
	with
	\begin{align*}
		\lt(\partial^{2k-1} x) &\doteq X(i,j')_{k}X(j',j'+1)_{k+1},\\
		\lt(\partial^{2k-2} x) &\doteq X(j',j')_{k}X(i,j'+1)_{k}.
	\end{align*}
	
	Lastly, if $i'+1<i$ and $i'+1<j$ (with $i\neq j$), then we start with \eqref{eqn:nondegtlinhshort}:
	\begin{align}
		&X(i'+1,j)_{1}X(i',i')_{1}
		+X(i'+1,i')_{1}X(i',j)_{1}\xrightarrow{X(i,i'+1) }\notag \\
		&x=X(i,j)_{1}X(i',i')_{1}-\ord{X(i'+1,j)_{1}}{2}\ord{X(i,i'+1)_{1}}{5}
		+\ord{X(i,i')_{1}}{6}\ord{X(i',j)_{1}}{1}\in \ring{T}.
		\label{eqn:nondegtlinh}
	\end{align}
	The ordering among the $X$ factors in the first quadratic depends on whether $i>j$ or not, but nonetheless, we can conclude:
	\begin{align*}
		\lt(\partial^{2k-1}x)&\doteq X(i,i')_{k}X(i',j)_{k+1},\\
		\lt(\partial^{2k-2}x)&\doteq X(i,j)_{k}X(i',i')_{k}.
	\end{align*}
	We include elements \eqref{eqn:nondegtlinhshort}, \eqref{eqn:nondegtlinhthin}, \eqref{eqn:nondegtlinh}, and their derivatives in $G_1$.
	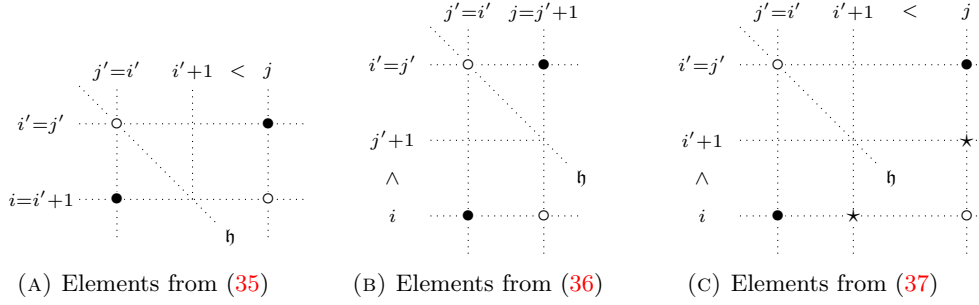
\begin{figure}
		\subcaptionbox{Elements from \eqref{eqn:nondegtlinhshort}}[0.3\linewidth]
		{
		\begin{tikzpicture}
			\draw[dotted] (1,0.5) -- (1,2.5);
			\draw[dotted] (2,1) -- (2,2.5);
			\draw[dotted] (3,0.5) -- (3,2.5);
			\draw[dotted] (0.5,1) -- (3.5,1);
			\draw[dotted] (0.5,2) -- (3.5,2);
			\draw[dotted] (0.5,2.5) -- (2.3,0.7);
			\node at (2.5,0.5) {$\scriptstyle{\la{h}}$};
			\node at (0,2) {$\scriptstyle{i'=j'}$};
			\node at (0,1) {$\scriptstyle{i=i'+1}$};
			\node at (1,2.7) {$\scriptstyle{j'=i'}$};
			\node at (2,2.7) {$\scriptstyle{i'+1}$};
			\node at (2.6,2.7) {$\scriptstyle{<}$};
			\node at (3,2.7) {$\scriptstyle{j}$};
			\node[fill=white,inner sep=0.3pt] at (1,1) {${\bullet}$};
			\node[fill=white,inner sep=0.3pt] at (3,2) {${\bullet}$};
			\node[fill=white,inner sep=0.3pt] at (3,1) {${\circ}$};
			\node[fill=white,inner sep=0.3pt] at (1,2) {${\circ}$};
		\end{tikzpicture}
		}
		\hfill
		\subcaptionbox{Elements from \eqref{eqn:nondegtlinhthin}}[0.3\linewidth]
		{
			\begin{tikzpicture}
				\draw[dotted] (1,0.5) -- (1,3.5);
				\draw[dotted] (2,0.5) -- (2,3.5);
				\draw[dotted] (0.5,1) -- (2.5,1);
				\draw[dotted] (0.5,2) -- (2,2);
				\draw[dotted] (0.5,3) -- (2.5,3);
				\draw[dotted] (0.5,3.5) -- (2.3,1.7);
				\node at (2.5,1.5) {$\scriptstyle{\la{h}}$};
				\node at (0,3) {$\scriptstyle{i'=j'}$};
				\node[rotate=-90] at (0,1.5) {$\scriptstyle{<}$};
				\node at (0,1) {$\scriptstyle{i}$};
				\node at (0,2) {$\scriptstyle{j'+1}$};
				\node at (1,3.7) {$\scriptstyle{j'=i'}$};
				\node at (2,3.7) {$\scriptstyle{j=j'+1}$};
				\node[fill=white,inner sep=0.3pt] at (1,1) {${\bullet}$};
				\node[fill=white,inner sep=0.3pt] at (2,3) {${\bullet}$};
				\node[fill=white,inner sep=0.3pt] at (2,1) {${\circ}$};
				\node[fill=white,inner sep=0.3pt] at (1,3) {${\circ}$};
			\end{tikzpicture}
		}
		\hfill
		\subcaptionbox{Elements from \eqref{eqn:nondegtlinh}}[0.3\linewidth]
		{
			\begin{tikzpicture}
				\draw[dotted] (1,0.5) -- (1,3.5);
				\draw[dotted] (2,0.5) -- (2,3.5);
				\draw[dotted] (3.5,0.5) -- (3.5,3.5);
				\draw[dotted] (0.5,1) -- (3.7,1);
				\draw[dotted] (0.5,2) -- (3.7,2);
				\draw[dotted] (0.5,3) -- (3.7,3);
				\draw[dotted] (0.5,3.5) -- (2.3,1.7);
				\node at (2.5,1.5) {$\scriptstyle{\la{h}}$};
				\node at (0,3) {$\scriptstyle{i'=j'}$};
				\node at (0,2) {$\scriptstyle{i'+1}$};
				\node[rotate=-90] at (0,1.5) {$\scriptstyle{<}$};
				\node at (0,1) {$\scriptstyle{i}$};
				\node at (1,3.7) {$\scriptstyle{j'=i'}$};
				\node at (2,3.7) {$\scriptstyle{i'+1}$};
				\node at (2.75,3.7) {$\scriptstyle{<}$};
				\node at (3.5,3.7) {$\scriptstyle{j}$};
				\node[fill=white,inner sep=0.3pt] at (1,1) {${\bullet}$};
				\node[fill=white,inner sep=0.3pt] at (3.5,3) {${\bullet}$};
				\node[fill=white,inner sep=0.3pt] at (2,1) {${\star}$};
				\node[fill=white,inner sep=0.3pt] at (3.5,2) {${\star}$};
				\node[fill=white,inner sep=0.3pt] at (1,3) {${\circ}$};
				\node[fill=white,inner sep=0.3pt] at (3.5,1) {${\circ}$};
			\end{tikzpicture}
		}
		\caption{Non-degenerate box, top-left in $\la{h}$}
		\label{fig:nondegtlinh}
	\end{figure}
	
	\subsection{Non-degenerate and left-bottom in \texorpdfstring{$\la{h}$}{h}}
	We start with the element obtained as the semi-degenerate horizontal element with left end-point on $\la{h}$ \eqref{eqn:semideghorizh} and commute as appropriate:
	\begin{align}
		&X(i,j)_{1}\sum_{i\leq s<j}X(i,i)_{1}\xrightarrow{X(i',i)}\notag\\
		&x=\ord{X(i',j)_{1}}{1}\sum_{i\leq s<j}\ord{X(s,s)_{1}}{\downarrow}+ \ord{X(i,j)_{1}}{3}\ord{X(i',i)_{1}}{2}\in \ring{T},
		\label{eqn:nondeglbinh}
	\end{align}
	where $\downarrow$ depicts factors lower than the factor marked $3$ (see Figure \ref{fig:nondeglbinh}).
	Comparing various leading terms using Lemma \ref{lem:quadder},
	\begin{align*}
		\lt(\partial^{2k-1}x) &\doteq X(i,i)_{k}X(i',j)_{k+1},\\
		\lt(\partial^{2k-2}x) &\doteq X(i',i)_{k}X(i,j)_{k}.
	\end{align*}
	We include in $G_1$  elements of type \eqref{eqn:nondeglbinh} and their derivatives.
	\begin{figure}
		\begin{tikzpicture}
			\draw[dotted] (0.5,2) -- (2.5,2);
			\draw[dotted] (0.5,1) -- (2.5,1);
			\draw[dotted] (1,-0.5) -- (1,2.5);
			\draw[dotted] (2,-0.5) -- (2,2.5);
			\draw[dotted] (0.5,1.5) -- (2.3,-0.3);
			\node at (2.5,-0.5) {$\scriptstyle{\la{h}}$};
			\node at (0,2) {$\scriptstyle{i'}$};
			\node[rotate=-90] at (0,1.5) {$\scriptstyle{<}$};
			\node at (0,1) {$\scriptstyle{i=j'}$};
			\node at (1,2.7) {$\scriptstyle{j'=i}$};
			\node at (1.6,2.7) {$\scriptstyle{<}$};
			\node at (2,2.7) {$\scriptstyle{j}$};
			\node[fill=white, inner sep=0.3pt] at (2,2) {$\bullet$};
			\node[fill=white, inner sep=0.3pt] at (1,1) {$\bullet$};
			\node[fill=white, inner sep=0.3pt] at (1,2) {$\circ$};
			\node[fill=white, inner sep=0.3pt] at (2,1) {$\circ$};
			\node[fill=white, inner sep=0.3pt] at (1.2,0.8) {\textcolor{gray}{$\bullet$}};
			\node[fill=white, inner sep=0.3pt] at (1.4,0.6) {\textcolor{gray}{$\bullet$}};
			\node[fill=white, inner sep=0.3pt] at (1.6,0.4) {\textcolor{gray}{$\bullet$}};
			\node[fill=white, inner sep=0.3pt] at (1.8,0.2) {\textcolor{gray}{$\bullet$}};
		\end{tikzpicture}			
		\caption{Non-degenerate and left-bottom in $\la{h}$}
		\label{fig:nondeglbinh}
	\end{figure}
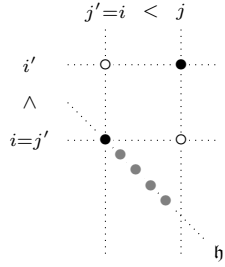
	
	\subsection{Non-degenerate and top-right in \texorpdfstring{$\la{h}$}{h}}
	We now assume that $i'=j$, and $i'<i$, $j'<j$.
	These elements are obtained by transposing the ones from the previous section.
	
	Thus, we get:
	\begin{align}
		&x=\ord{X(i,j')_{1}}{4}\sum_{i'\leq s<i}\ord{X(s,s)_{1}}{\uparrow} + \ord{X(i,j)_{1}}{2}\ord{X(i',j')_{1}}{3}\in \ring{T},
		\label{eqn:nondegtrinh}
	\end{align}
	where $\uparrow$ depicts factors higher than the factor marked $2$.
	For the ordering of elements, see Figure \ref{fig:nondegtrinh}.
	We see:
	\begin{align*}
		\lt(\partial^{2k-1}x) &\doteq X(i,j')_{k}X(i',i')_{k+1},\\
		\lt(\partial^{2k-2}x) &\doteq X(i,j)_{k}X(i',j')_{k}.
	\end{align*}
	We include in $G_1$ elements \eqref{eqn:nondegtrinh} and its derivatives.
	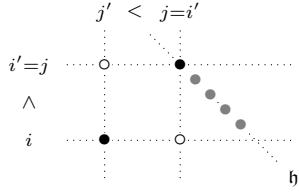
\begin{figure}
		\begin{tikzpicture}
			\draw[dotted] (0.5,2) -- (3.5,2);
			\draw[dotted] (0.5,1) -- (3.5,1);
			\draw[dotted] (1,0.5) -- (1,2.5);
			\draw[dotted] (2,0.5) -- (2,2.5);
			\draw[dotted] (1.6,2.4) -- (3.3,0.7);
			\node at (3.5,0.5) {$\scriptstyle{\la{h}}$};
			\node at (0,2) {$\scriptstyle{i'=j}$};
			\node[rotate=-90] at (0,1.5) {$\scriptstyle{<}$};
			\node at (0,1) {$\scriptstyle{i}$};
			\node at (1,2.7) {$\scriptstyle{j'}$};
			\node at (1.4,2.7) {$\scriptstyle{<}$};
			\node at (2,2.7) {$\scriptstyle{j=i'}$};
			\node[fill=white, inner sep=0.3pt] at (2,2) {$\bullet$};
			\node[fill=white, inner sep=0.3pt] at (1,1) {$\bullet$};
			\node[fill=white, inner sep=0.3pt] at (1,2) {$\circ$};
			\node[fill=white, inner sep=0.3pt] at (2,1) {$\circ$};
			\node[fill=white, inner sep=0.3pt] at (2.2,1.8) {\textcolor{gray}{$\bullet$}};
			\node[fill=white, inner sep=0.3pt] at (2.4,1.6) {\textcolor{gray}{$\bullet$}};
			\node[fill=white, inner sep=0.3pt] at (2.6,1.4) {\textcolor{gray}{$\bullet$}};
			\node[fill=white, inner sep=0.3pt] at (2.8,1.2) {\textcolor{gray}{$\bullet$}};
		\end{tikzpicture}			
		\caption{Non-degenerate and top-right in $\la{h}$}
		\label{fig:nondegtrinh}
	\end{figure}
	
	\subsection{Summary of elements obtained thus far} For the boxes considered up to now,  we have produced elements of $\ring{T}$ such that their odd order derivatives have as leading monomials quadratics whose factors lie on the top-right and bottom-left corners as in Figure \ref{fig:boxes}. This pattern will continue to hold below.

	Crucially, the leading monomial of even order derivative of such an  element is a quadratic whose factors lie on the top-left and bottom-right corners of the \emph{same} box. There will be several exceptions to this pattern below.
	
	\subsection{Non-degenerate and bottom-right in \texorpdfstring{$\la{h}^e$}{h e}} Lastly, we consider the case where $i=j$ (and we allow $i=j=n$), but that $i'\neq j'$, $i'<i$, $j'<j$. See Figure \ref{fig:nondegbrinhe} for this section.

	First, if $i=j=i'+1$, we start with the element obtained in \eqref{eqn:semideghorizh} with indices adjusted. Note that $j'\neq i'$ and $j'<j=i'+1$ implies that $j'<i'$. 
	Then,
	\begin{align}
		&X(i',j')_{1}X(i',i'+1)_{1}\xrightarrow{X(i'+1,i')}\notag\\
		&x=\ord{X(i'+1,j')_{1}}{4}
		\ord{X(i',i'+1)_{1}}{1} -  \ord{X(i',j')_{1}}{3}\ord{X(i',i')_{1}}{2}\in \ring{T},
		\label{eqn:nondegbrinheshort}
	\end{align}
	Note that here $1\leq i'\leq n-1$ since $i=i'+1$ and $i\leq n$.
	We have:
	\begin{align*}
		\lt(\partial^{2k-1} x) &\doteq X(i'+1,j')_{k}X(i',i'+1)_{k+1},\\
		\lt(\partial^{2k-2} x) &\doteq X(i',i')_{k}X(i',j')_{k+1}.
	\end{align*}
	\subsubsection*{Summary} Here we observe an interesting phenomenon. If we have a non-degenerate box of height $1$ with bottom-right corner in $\la{h}^e$ (and no other corners on $\la{h}^e$), then the  even ordered derivatives of the corresponding elements have as leading monomials quadratics that lie on semi-degenerate horizontal boxes with right end points on $\la{h}$.
	
	If $i=j=j'+1$, we just use the transpose of the element obtained above, but restore the indices as in Figure \ref{fig:iji'j'}:
	\begin{align}
		x=\ord{X(i',j'+1)_{1}}{1}
		\ord{X(j'+1,j')_{1}}{4} -  \ord{X(i',j')_{1}}{2}\ord{X(j',j')_{1}}{3}\in \ring{T},
		\label{eqn:nondegbrinhethin}
	\end{align}
	and 
	\begin{align*}
		\lt(\partial^{2k-1} x) \doteq X(j'+1,j')_{k}X(i',j'+1)_{k+1},\\
		\lt(\partial^{2k-2} x) \doteq X(i',j')_{k}X(j',j')_{k}.
	\end{align*}		
	Here as well, $1\leq j'\leq n-1$ since $j=j'+1$ and $j\leq n$.

	\subsubsection*{Summary} If we have a non-degenerate box of width $1$ with bottom-right corner in $\la{h}^e$ (and no other corners on $\la{h}^e$) then the corresponding elements of  $\ring{T}$ we have produced have the following property.
	The leading monomials of their even ordered derivatives give rise to quadratics whose factors lie on semi-degenerate vertical boxes with bottom end points on $\la{h}$.

	Lastly, if $i'+1<i=j$ and $j'+1<i=j$, we start with $X(i',j')_{1}X(i',i)_{1}\in T$.
	\begin{align}
		&X(i',j')_{1}X(i',i)_{1}\xrightarrow{X(i-1,i')}
		X(i-1,j')_{1}X(i',i)_{1}+X(i',j')_{1}X(i-1,i)_{1}\xrightarrow{X(i,i-1)}\notag\\
		&x=\ord{X(i,j')_{1}}{6}
		\ord{X(i',i)_{1}}{1} - 
		\ord{X(i-1,j')_{1}}{5}
		\ord{X(i',i-1)_{1}}{2} -X(i',j')_{1}
		X(i-1,i-1)_{1}\in \ring{T}
		\label{eqn:nondegbrinhe}
	\end{align}
	The order of the factors in the last term depends on whether $i'>j'$ or not. Regardless, we see:
	\begin{align*}
		\lt(\partial^{2k-1} x) &\doteq X(i,j')_{k}X(i',i)_{k+1},\\
		\lt(\partial^{2k-2} x) &\doteq X(i',j')_{k}X(i-1,i-1)_{k}.
	\end{align*}
	We include in $G_1$ elements \eqref{eqn:nondegbrinheshort}, \eqref{eqn:nondegbrinhethin}, \eqref{eqn:nondegbrinhe}, and their derivatives.
	\subsubsection*{Summary} Elements corresponding to non-degenerate boxes whose bottom-right corners (and no other corners) are on $\la{h}^e$ and whose height and width are at least $2$ have the following property. The totality of leading monomials of their even ordered derivatives correspond to top-left and bottom-right corners of non-degenerate boxes potentially having width or height $1$ whose bottom-right corner (and no other corner) is in $\la{h}$.
	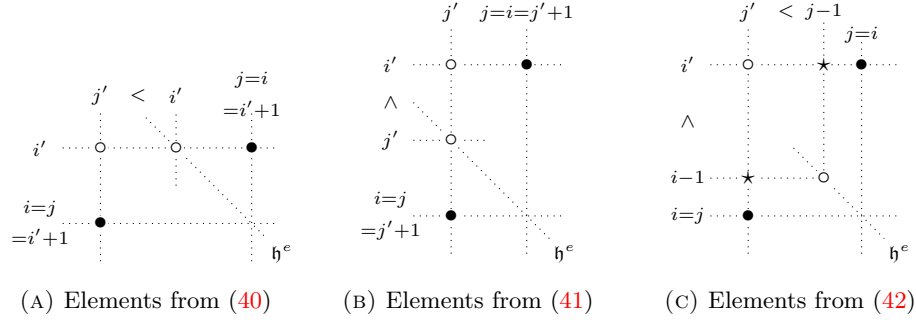
\begin{figure}
		\subcaptionbox{Elements from \eqref{eqn:nondegbrinheshort}}[0.3\linewidth]
		{
			\begin{tikzpicture}
				\draw[dotted] (0.5,1) -- (3.3,1); 
				\draw[dotted] (0.5,2) -- (3.3,2);
				\draw[dotted] (1,0.5) -- (1,2.5);
				\draw[dotted] (2,1.5) -- (2,2.5);
				\draw[dotted] (3,0.5) -- (3,2.5);
				\draw[dotted] (1.6,2.4) -- (3.2,0.8);
				\node at (3.4,0.6) {$\scriptstyle{\la{h}^e}$};
				\node at (0.2,2) {$\scriptstyle{i'}$};
				\node[align=center] at (0.2,1) {$\scriptstyle{i=j}$\\ $\scriptstyle{=i'+1}$};
				\node at (1,2.7) {$\scriptstyle{j'}$};
				\node at (1.5,2.7) {$\scriptstyle{<}$};
				\node at (2,2.7) {$\scriptstyle{i'}$};
				\node[align=center] at (3,2.7) {$\scriptstyle{j=i}$\\$\scriptstyle{=i'+1}$};
				\node[fill=white, inner sep=0.3pt] at (3,2) {$\bullet$};
				\node[fill=white, inner sep=0.3pt] at (1,1) {$\bullet$};
				\node[fill=white, inner sep=0.3pt] at (1,2) {$\circ$};
				\node[fill=white, inner sep=0.3pt] at (2,2) {$\circ$};
			\end{tikzpicture}					
		}			
		\quad
		\subcaptionbox{Elements from \eqref{eqn:nondegbrinhethin}}[0.3\linewidth]
		{
			\begin{tikzpicture}
				\draw[dotted] (0.5,1) -- (1.5,1); 
				\draw[dotted] (0.5,2) -- (2.5,2);
				\draw[dotted] (0.5,0) -- (2.5,0);
				\draw[dotted] (1,-0.5) -- (1,2.5);
				\draw[dotted] (2,-0.5) -- (2,2.5);
				\draw[dotted] (0.5,1.5) -- (2.3,-0.3);
				\node at (2.5,-0.5) {$\scriptstyle{\la{h}^e}$};
				\node at (0.2,2) {$\scriptstyle{i'}$};
				\node[rotate=-90] at (0.2,1.5) {$\scriptstyle{<}$};
				\node at (0.2,1) {$\scriptstyle{j'}$};
				\node[align=center] at (0.2,0) {$\scriptstyle{i=j}$\\$\scriptstyle{=j'+1}$};
				\node at (1,2.7) {$\scriptstyle{j'}$};
				\node at (2,2.7) {$\scriptstyle{j=i=j'+1}$};
				\node[fill=white, inner sep=0.3pt] at (2,2) {$\bullet$};
				\node[fill=white, inner sep=0.3pt] at (1,0) {$\bullet$};
				\node[fill=white, inner sep=0.3pt] at (1,2) {$\circ$};
				\node[fill=white, inner sep=0.3pt] at (1,1) {$\circ$};
			\end{tikzpicture}				
		}			
		\quad 
		\subcaptionbox{Elements from \eqref{eqn:nondegbrinhe}}[0.3\linewidth]
		{
			\begin{tikzpicture}
				\draw[dotted] (0.5,0.5) -- (2,0.5)--(2,2.6); 
				\draw[dotted] (0.5,2) -- (3,2);
				\draw[dotted] (0.5,0) -- (3,0);
				\draw[dotted] (1,-0.5) -- (1,2.5);
				\draw[dotted] (2.5,-0.5) -- (2.5,2.3);
				\draw[dotted] (1.6,0.9) -- (2.8,-0.3);
				\node at (3,-0.5) {$\scriptstyle{\la{h}^e}$};
				\node at (0.2,2) {$\scriptstyle{i'}$};
				\node[rotate=-90] at (0.2,1.25) {$\scriptstyle{<}$};
				\node at (0.2,0.5) {$\scriptstyle{i-1}$};
				\node[align=center] at (0.2,0) {$\scriptstyle{i=j}$};
				\node at (1,2.7) {$\scriptstyle{j'}$};
				\node at (1.5,2.7) {$\scriptstyle{<}$};
				\node at (2,2.7) {$\scriptstyle{j-1}$};
				\node at (2.5,2.4) {$\scriptstyle{j=i}$};
				\node[fill=white, inner sep=0.3pt] at (2.5,2) {$\bullet$};
				\node[fill=white, inner sep=0.3pt] at (1,0) {$\bullet$};
				\node[fill=white, inner sep=0.3pt] at (1,2) {$\circ$};
				\node[fill=white, inner sep=0.3pt] at (2,0.5) {$\circ$};
				\node[fill=white, inner sep=0.3pt] at (1,0.5) {$\star$};
				\node[fill=white, inner sep=0.3pt] at (2,2) {$\star$};
			\end{tikzpicture}				
		}			
		\caption{Non-degenerate box, bottom-right corner in $\la{h}^e$}
		\label{fig:nondegbrinhe}
	\end{figure}

	\subsection{Two locations on \texorpdfstring{$\la{h}^e$}{h e}}
	Lastly, we take care of two locations being on $\la{h}^e$. So, we assume that $i'=j'<i=j$. Figure \ref{fig:twoinhe} refers to this section.
	
	If $i=i'+1$ (with $i=n$ being allowed), we start with the element obtained in \eqref{eqn:semideghorizh} where the box has height $1$ and  adjust the indices:
	\begin{align}
		&X(i',i'+1)_{1}X(i',i')_{1}\xrightarrow{X(i'+1,i')} \notag\\
		&x=- 
		\underbracket{X(i',i')_{1}}_{\circnum{2}}
		\underbracket{X(i',i')_{1}}_{\circnum{2}}+
		2\underbracket{X(i',i'+1)_{1}}_{\circnum{1}}
		\underbracket{X(i'+1,i')_{1}}_{\circnum{3}} \in \ring{T}.
		\label{eqn:twoinhesmall}
	\end{align}
	We thus have that:
	\begin{align*}
		\lt(\partial^{2k-1} x) &\doteq X(i'+1,i')_{k}X(i',i'+1)_{k+1},\\
		\lt(\partial^{2k-2} x) &\doteq X(i',i')_{k}X(i',i')_{k}.
	\end{align*}

	However, if $i'+1<i$ (here, $i=n$ is allowed), we need a more complicated element.
	We start with an element of the type obtained in \eqref{eqn:nondeglbinh} with indices adjusted:
	\begin{align}
		&X(i',i-1)_{1}X(i-1,i)_{1}+X(i',i)_{1}X(i-1,i-1)_{1}\notag \\
		&\xrightarrow{X(i-1,i')}
		(E(i-1,i-1)_{1}-E(i',i')_{1})X(i-1,i)_{1}\notag\\
		&\quad+ X(i-1,i)_{1}X(i-1,i-1)_{1}
		- X(i',i)_{1}X(i-1,i')_{1}\notag\\
		&\xrightarrow{X(i,i-1)}
		X(i,i-1)_{1}X(i-1,i)_{1}
		-(E(i-1,i-1)_{1}-E(i',i')_{1})X(i-1,i-1)_{1}\notag\\
		&\quad-X(i-1,i-1)_{1}X(i-1,i-1)_{1}
		+2X(i-1,i)_{1}X(i,i-1)_{1}\notag\\
		&\quad +X(i',i-1)_{1}X(i-1,i')_{1}-X(i',i)_{1}X(i,i')_{1}\notag\\
		&= - X(i-1,i-1)_{1}^2 + \sum_{i'\leq  s< i-1}X(s,s)_{1}X(i-1,i-1)_{1}  
		+ 3X(i,i-1)_{1}X(i-1,i)_{1}\notag\\
		&\quad+X(i',i-1)_{1}X(i-1,i')_{1} -\underbracket{X(i',i)_{1}}_{\circnum{\uparrow}}
		\underbracket{X(i,i')_{1}}_{\circnum{\downarrow}} = x \in \ring{T}.
		\label{eqn:twoinhe}
	\end{align}
	Where the up and down arrows signify the highest and lowest $X$ elements, respectively, in the entire expression.
	It is not too hard to see that:
	\begin{align*}
		\lt(\partial^{2k-1}x)&\doteq X(i,i')_{k}X(i',i)_{k+1},\\
		\lt(\partial^{2k-2}x)&\doteq X(i',i')_{k}X(i-1,i-1)_{k}.
	\end{align*}
	Finally, we enlarge $G_1$ by including elements \eqref{eqn:twoinhesmall}, \eqref{eqn:twoinhe}, and their derivatives of all orders. This concludes our construction of $G_1$.
	\begin{figure}
		\subcaptionbox{Elements from \eqref{eqn:twoinhesmall}}[0.43\linewidth]
		{
			\begin{tikzpicture}
				\draw[dotted] (0.5,1) -- (2.5,1);
				\draw[dotted] (0.5,2) -- (2.5,2);
				\draw[dotted] (1,0.5) -- (1,2.5);
				\draw[dotted] (2,0.5) -- (2,2.5);
				\draw[dotted] (0.5,2.5) -- (2.2,0.8);
				\node at (2.5,0.5) {$\scriptstyle{\la{h}^e}$};
				\node at (0,2) {$\scriptstyle{i'}$};
				\node at (0,1) {$\scriptstyle{i=i'+1}$};
				\node at (1,3) {$\scriptstyle{j'=i'}$};
				\node[align=center] at (2,3) {$\scriptstyle{j=i}$\\$\scriptstyle{=i'+1}$};
				\node[fill=white,inner sep=0.3pt] at (2,2) {${\bullet}$};
				\node[fill=white,inner sep=0.3pt] at (1,1) {${\bullet}$};
				\node[fill=white,inner sep=0.3pt] at (1,2) {${\circ}$};
			\end{tikzpicture}
		}			
		\subcaptionbox{Elements from \eqref{eqn:twoinhe}}[0.43\linewidth]
		{
				\begin{tikzpicture}
				\draw[dotted] (0.5,1) -- (3,1);
				\draw[dotted] (0.5,0.5) -- (3,0.5);
				\draw[dotted] (0.5,2) -- (3,2);
				\draw[dotted] (1,0) -- (1,2.7);
				\draw[dotted] (2,0) -- (2,2.7);
				\draw[dotted] (2.5,0) -- (2.5,2.3);
				\draw[dotted] (0.5,2.5) -- (2.8,0.2);
				\node at (3,0) {$\scriptstyle{\la{h}^e}$};
				\node at (0.2,2) {$\scriptstyle{i'}$};
				\node[rotate=-90] at (0.2,1.5) {$\scriptstyle{<}$};
				\node at (0.2,1) {$\scriptstyle{i-1}$};
				\node at (0.2,0.5) {$\scriptstyle{i}$};
				\node at (1,3) {$\scriptstyle{j'=i'}$};
				\node at (1.5,3) {$\scriptstyle{<}$};
				\node[align=center] at (2.5,2.5) {$\scriptstyle{j=i}$};
				\node[align=center] at (2,3) {$\scriptstyle{i-1}$};
				\node[fill=white,inner sep=0.3pt] at (2.5,2) {${\bullet}$};
				\node[fill=white,inner sep=0.3pt] at (1,0.5) {${\bullet}$};
				\node[fill=white,inner sep=0.3pt] at (1,2) {${\circ}$};
				\node[fill=white,inner sep=0.3pt] at (2,1) {${\circ}$};
				\node[fill=white,inner sep=0.3pt] at (2.5,1) {${\times}$};
				\node[fill=white,inner sep=0.3pt] at (2,0.5) {${\times}$};
				\node[fill=white,inner sep=0.3pt] at (2,2) {${\star}$};
				\node[fill=white,inner sep=0.3pt] at (1,1) {${\star}$};
				\node[fill=white,inner sep=0.3pt] at (1.2,1.8) {$\textcolor{gray}{\circ}$};
				\node[fill=white,inner sep=0.3pt] at (1.4,1.6) {$\textcolor{gray}{\circ}$};
				\node[fill=white,inner sep=0.3pt] at (1.6,1.4) {$\textcolor{gray}{\circ}$};
				\node[fill=white,inner sep=0.3pt] at (1.8,1.2) {$\textcolor{gray}{\circ}$};
			\end{tikzpicture}
		}			
		\caption{Two locations in $\la{h}^e$}
		\label{fig:twoinhe}
	\end{figure}
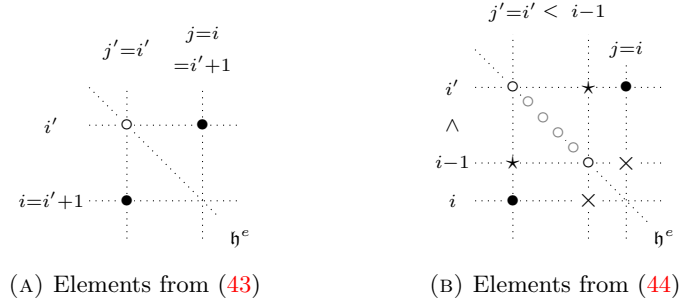

	\subsubsection*{Summary} In this case, we again observe that leading terms of odd ordered and even ordered derivatives of the elements do not always correspond to the corners of the same box.  However, if we consider totality of elements obtained, then the leading terms of their odd ordered derivatives do exhaust non-degenerate boxes whose two corners are on $\la{h}^e$. The totality of leading terms of even ordered derivatives of these elements correspond to top-left and bottom-right corners of possibly degenerate boxes whose two corners are in $\la{h}$.

	Now that we have found enough elements of $\ring{T}$, we make sure that they form a basis of $\ring{T}$. 
	
	\begin{lem}
		We have $\dim(\ring{T})=\frac{1}{4}n^2(n+3)(n-1)$.
	\end{lem}
	\begin{proof}
		Recall that $\ring{T}\cong L(2\theta)$.
		Note that $2\theta=2(\alpha_1+\cdots+\alpha_{n-1})=2\Lambda_1+2\Lambda_{n-1}$ in terms of fundamental weights.
		The only positive roots which pair non-trivially with $2\theta$ are thus:
		$$ S=  \{ \alpha_1+\cdots+\alpha_j\,|\, 1\leq j<n-1\} \cup \{\alpha_1+\cdots+\alpha_{n-1}\}\cup
		\{ \alpha_j+\cdots+\alpha_{n-1} \,|\, 1<j\leq n-1\}.$$
		Thus, by Weyl's dimension formula,
		\begin{align*}
			\dim(\ring{T}) = 
			\prod_{\alpha\in S}\frac{\langle 2\Lambda_1+ 2\Lambda_{n-1}+\rho ,\alpha\rangle}{\langle \rho ,\alpha\rangle}
			=\dfrac{ \left(\prod_{j=1}^{n-2} (j+2)\right) (n+3)\left(\prod_{j=2}^{n-1} (n-j+2)\right) }{\left(\prod_{j=1}^{n-2} j\right) (n-1)\left(\prod_{j=2}^{n-1} (n-j)\right)},
		\end{align*}
		which can be easily seen to equal the required quartic.
	\end{proof}
	
	\begin{prop}
		The elements found in this section form a basis of $\ring{T}$.
	\end{prop}
	\begin{proof}
		Leading terms of derivatives of the elements we have found are easily seen to be linearly independent (each corresponds to a unique bounding box). So, we just need to count the number of bounding boxes we have considered.
		
		Without considering any exclusions, number of bounding boxes where the left-top corner is in row $i$ and column $j$ is $(n-i+1)(n-j+1)$. Summing over all $1\leq i\leq n$ and $1\leq j\leq n$ gets us $\frac{1}{4}n^2(n+1)^2$.
		
		We have excluded $n$ many totally degenerate bounding boxes lying on $\la{h}^e$. The  number semi-degenerate horizontal boxes which start at the diagonal element $(i,i)$ ($1\leq i\leq n-1$) is $n-i$. So, the total number of such horizontal boxes excluded is $\frac{1}{2}n(n-1)$ and the same number of vertical boxes are also excluded.
		
		Overall, the total number of bounding boxes considered is
		\begin{align*}
			\dfrac{1}{4}n^2(n+1)^2	- n - n(n-1) =\frac{1}{4}n^2(n+3)(n-1),
		\end{align*}
		which is exactly the dimension of $\ring{T}$.
	\end{proof}

	\subsection{Summary of quadratic leading monomials of \texorpdfstring{$G_1$}{G1}}
	\label{sec:ltG1}
	If we consider the chosen basis elements of $\ring{T}$ which we have obtained in this section, we observe the following.

	The totality of leading terms of their odd ordered derivatives correspond to the top-right and bottom-left corners of the boxes considered at the very beginning of the section, see Figure \ref{fig:iji'j'}.

	The totality of leading terms of their even ordered derivatives correspond to the top-left and bottom-right corners of all boxes with the exception that the bottom-right corner is not $X(n,n)$, i.e., boxes corresponding to indices $1\leq i'\leq i\leq n$, $1\leq j'\leq j\leq n$, with $(i,j)\neq (n,n)$, see Figure \ref{fig:boxes2}.
	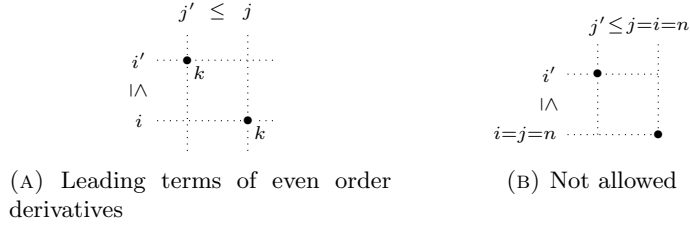
\begin{figure}
		\subcaptionbox{Leading terms of even order derivatives}[0.4\linewidth]{
			\begin{tikzpicture}[scale=0.8]
				\draw[dotted] (0.5,1) -- (2.5,1);
				\draw[dotted] (0.5,2) -- (2.5,2);
				\draw[dotted] (1,0.5) -- (1,2.5);
				\draw[dotted] (2,0.5) -- (2,2.5);
				\node at (0.2,1) {$\scriptstyle{i}$};
				\node[rotate=-90] at (0.2,1.5) {$\scriptstyle{\leq}$};
				\node at (0.2,2) {$\scriptstyle{i'}$};
				\node at (1,2.8) {$\scriptstyle{j'}$};
				\node at (1.5,2.8) {$\scriptstyle{\leq}$};
				\node at (2,2.8) {$\scriptstyle{j}$};
				\node[fill=white, inner sep=0.3pt] at (1,2) {$\scriptstyle{\bullet}$};
				\node[fill=white, inner sep=0.3pt] at (2,1) {$\scriptstyle{\bullet}$};
				\node at (1.2,1.8) {$\scriptstyle{k}$};
				\node at (2.2,0.8) {$\scriptstyle{k}$};
			\end{tikzpicture}
		}
		\subcaptionbox{Not allowed}[0.4\linewidth]{
			\begin{tikzpicture}[scale=0.8]
			\draw[dotted] (0.5,1) -- (2,1);
			\draw[dotted] (0.5,2) -- (2,2);
			\draw[dotted] (1,1) -- (1,2.5);
			\draw[dotted] (2,1) -- (2,2.5);
			\node at (-0.2,1) {$\scriptstyle{i=j=n}$};
			\node[rotate=-90] at (0.2,1.5) {$\scriptstyle{\leq}$};
			\node at (0.2,2) {$\scriptstyle{i'}$};
			\node at (1,2.8) {$\scriptstyle{j'}$};
			\node at (1.3,2.8) {$\scriptstyle{\leq}$};
			\node at (2,2.8) {$\scriptstyle{j=i=n}$};
			\node[fill=white, inner sep=0.3pt] at (1,2) {$\scriptstyle{\bullet}$};
			\node[fill=white, inner sep=0.3pt] at (2,1) {$\scriptstyle{\bullet}$};
			\end{tikzpicture}
		}
		\caption{Leading terms of even order derivatives}
		\label{fig:boxes2}
	\end{figure}

	\section{\texorpdfstring{$S$}{S}-polynomials}
	\label{sec:cubic}
	Let $n\geq 3$. If $n=2$, we are in the $\la{sl}_2$ case, and our chosen basis along with derivatives already forms a Gr\"obner basis \cite{vanEkeHel-chiralhom}.
	
	We now demonstrate certain $S$-polynomials which do not reduce to $0$ modulo  our set $G_1$.
	This constitutes our second step in the Buchberger's algorithm.
	To this end, we let
	\begin{align*}
		G_2=G_1,
	\end{align*}
	and we shall enlarge this set as we go along.
	
	\begin{nota}
		Ellipses $\cdots$ in various expressions will denote elements that are lower in the order. Elements that have lower shape (in particular, they are also lower in the order) will be denoted by $\cdots_{\sh}$. 
	\end{nota}

	Throughout the rest of the section, $k\in\ZZ_{\geq 1}$.

	\subsection{Weight \texorpdfstring{$3k+1$}{3k+1}}
	We first tackle the $S$-polynomials with weight $3k+1$.
	For this purpose, choose indices $i,j,i',j',l$ such that (see Figure \ref{fig:3k+1}):
	\begin{align}
		1\leq i' \leq l<i=j' < j\leq n.
		\label{eqn:wt3k+1indices}
	\end{align} 
	In this section, our aim is to show that:
	$$X(l,j)_{k}X(i,i)_{k}X(i',j')_{k+1}\in \LT(\ring{I}).$$
	It is clear that such a leading term is indivisible by any of the quadratic leading terms of $G_1$ (see Section \ref{sec:ltG1}). We consider two cases: $i'=l$ and $i'<l$.
	
	\begin{figure}
		\subcaptionbox{Leading of weight $3k+1$	\label{fig:3k+1}}[0.43\linewidth]
		{
			\begin{tikzpicture}
				\draw[dotted] (0.5,0.5) -- (2.5,0.5);
				\draw[dotted] (0.5,1.5) -- (2.5,1.5);
				\draw[dotted] (0.5,2.5) -- (2.5,2.5);
				\draw[dotted] (1,0) -- (1,3);
				\draw[dotted] (2,0) -- (2,3);
				\draw[dotted] (0.5,1) -- (1.5,0);
				\node at (1.7,-0.2) {$\scriptstyle{\la{h}}$};
				\node at (0,0.5) {$\scriptstyle{i=j'}$};
				\node[rotate=-90] at (0,1) {$\scriptstyle{<}$};
				\node at (0,1.5) {$\scriptstyle{l}$};
				\node[rotate=-90] at (0,2) {$\scriptstyle{\leq}$};
				\node at (0,2.5) {$\scriptstyle{i'}$};
				\node at (1,3.2) {$\scriptstyle{j'=i}$};
				\node at (1.5,3.2) {$\scriptstyle{<}$};
				\node at (2,3.2) {$\scriptstyle{j}$};
				\node[fill=white, inner sep=0.3pt] at (1,0.5) {$\scriptstyle{\bullet}$};
				\node[fill=white, inner sep=0.3pt] at (2,1.5) {$\scriptstyle{\bullet}$};
				\node[fill=white, inner sep=0.3pt] at (1,2.5) {$\scriptstyle{\bullet}$};
				\node at (1.2,0.7) {$\scriptstyle{k}$};
				\node at (2.2,1.3) {$\scriptstyle{k}$};
				\node at (1.3,2.3) {$\scriptstyle{k+1}$};
			\end{tikzpicture}
		}
		\subcaptionbox{Leading terms of weight $3k+2$\label{fig:3k+2}}[0.43\linewidth]
		{
			\begin{tikzpicture}
				\draw[dotted] (0.5,0.5) -- (0.5,2.5);
				\draw[dotted] (1.5,0.5) -- (1.5,2.5);
				\draw[dotted] (2.5,0.5) -- (2.5,2.5);
				\draw[dotted] (0,1)--(3,1);
				\draw[dotted] (0,2)--(3,2);
				\draw[dotted] (2,2.5) -- (3,1.5);
				\node at (3.2,1.3) {$\scriptstyle{\la{h}}$};
				\node at (-0.4,1) {$\scriptstyle{i}$};
				\node[rotate=-90] at (-0.4,1.5) {$\scriptstyle{<}$};
				\node at (-0.4,2) {$\scriptstyle{i'=j}$};
				\node at (0.5,2.7) {$\scriptstyle{j'}$};
				\node at (1,2.7) {$\scriptstyle{\leq}$};
				\node at (1.5,2.7) {$\scriptstyle{l}$};
				\node at (2,2.7) {$\scriptstyle{<}$};
				\node at (2.5,2.7) {$\scriptstyle{j=i'}$};
				\node[fill=white, inner sep=0.3pt] at (0.5,2) {$\scriptstyle{\bullet}$};
				\node[fill=white, inner sep=0.3pt] at (1.5,1) {$\scriptstyle{\bullet}$};
				\node[fill=white, inner sep=0.3pt] at (2.5,2) {$\scriptstyle{\bullet}$};
				\node at (0.7,1.8) {$\scriptstyle{k}$};
				\node at (1.1,0.8) {$\scriptstyle{k+1}$};
				\node at (2.1,1.8) {$\scriptstyle{k+1}$};
			\end{tikzpicture}
		}
		\caption{Cubic leading terms in $\LT(\ring{I})$}
		\label{fig:cubic}
	\end{figure}

	\subsubsection{Weight \texorpdfstring{$3k+1$}{3k+1} and \texorpdfstring{$i'=l$}{i'=l}}
	In this section, we assume that $i'=l$.
	
	The first element we consider is a derivative of a horizontal element from \eqref{eqn:semideghoriz}:
	\begin{align*}
		A&=\partial^{[2k-1]}
		(\ord{X(i',j)_{1}}{1}\ord{X(i',j')_{1}}{2})\\
		&=X(i',j')_{k}X(i',j)_{k+1}+X(i',j)_{k}X(i',j')_{k+1}+\cdots_{\sh}
	\end{align*}
 	Next, consider the derivative of the element corresponding to a non-degenerate box with bottom-left corner in $\la{h}$ \eqref{eqn:nondeglbinh}:
	\begin{align*}
		&B=\partial^{[2k-1]}
		\left( \ord{X(i',j')_{1}}{2}\ord{X(i,j)_{1}}{3} +\ord{X(i',j)_{1}}{1}\sum_{i\leq s<j}\ord{X(s,s)_{1}}{\downarrow}\right)\\
		&=X(i,i)_{k}X(i',j)_{k+1}+\sum_{i<s<j}X(s,s)_{k}X(i',j)_{k+1}
		+X(i,j)_{k}X(i',j')_{k+1}\\
		&+X(i',j')_{k}X(i,j)_{k+1}
		+X(i',j)_{k}\sum_{i\leq s<j}X(s,s)_{k+1}
		+\cdots_{\sh},
	\end{align*}
	where we have written the terms in descending order.
	Thus, we see that:
	\begin{align*}
		S(A,B)=&
		-X(i',j')_{k}\sum_{i<s<j}X(s,s)_{k}X(i',j)_{k+1}\\
		&-X(i',j')_{k}X(i,j)_{k}X(i',j')_{k+1}\\
		&+X(i',j)_{k}X(i,i)_{k}X(i',j')_{k+1}\\
		&-X(i',j')_{k}X(i',j')_{k}X(i,j)_{k+1}\\
		&-X(i',j')_{k}X(i',j)_{k}\sum_{i\leq s<j}X(s,s)_{k+1}+\cdots_{\sh},
	\end{align*}
	where it is easy to observe that the terms appear in descending order.
	Our aim is to reduce this $S$-polynomial so that the third term is the highest.
	So, we must eliminate the first and second terms.
	Starting with an element of the type \eqref{eqn:semideghoriz}, we obtain the following element, whose leading monomial is the same as the leading monomial of the $S$-polynomial above:
	\begin{align*}
		&\left(\partial^{[2k-1]}(\ord{X(i',j)_{1}}{1}\ord{X(i',j')_{1}}{2})\right) \left(\sum_{i<s<j}\ord{X(s,s)_{k}}{\downarrow}\right)\\
		&=X(i',j')_{k}\left(\sum_{i<s<j}X(s,s)_{k}\right)X(i',j)_{k+1}\\
		&+X(i',j)_{k}\left(\sum_{i<s<j}X(s,s)_{k}\right)X(i',j')_{k+1}+\cdots_{\sh}.
	\end{align*}
	All terms except the first in this element are strictly lower than the third term in the $S$-polynomial above. So, using this element, we get the reduction:
	\begin{align}
		S(A,B)\xrightarrow{G_1}&
		-X(i',j')_{k}X(i,j)_{k}X(i',j')_{k+1}
		+X(i',j)_{k}X(i,i)_{k}X(i',j')_{k+1}+\cdots.
		\label{eqn:3k+1_l=i_S_step2}
	\end{align}
	We reduce the first term here
	using an element of the type obtained in  \eqref{eqn:deg}:
	\begin{align*}
		\left(\partial^{[2k-1]}(X(i',j')_{1}^2)\right) X(i,j)_k
		=2 X(i',j')_{k}X(i,j)_kX(i',j')_{k+1}+
		\cdots_{\sh}.
	\end{align*}
	
	This reduces our $S$-polynomial to:
	\begin{align}
		S(A,B)\xrightarrow{G_1}&X(i',j)_{k}X(i,i)_{k}X(i',j')_{k+1}+\cdots.
		\label{eqn:S3k+1l=i'red}
	\end{align}
	We enlarge $G_2$ by including this reduction in it.

	\subsubsection{Weight \texorpdfstring{$3k+1$}{3k+1} and \texorpdfstring{$i'<l$}{i'<l}} Our strategy is very similar to before, except that the analogue of the element $A$ is now slightly more complicated.
	
	Constraints of \eqref{eqn:wt3k+1indices} are in force.
	
	We start with derivative of an element of the type \eqref{eqn:nondegnohe}:
	\begin{align*}
		A=&\partial^{[2k-1]}
		\left(
		\ord{X(l,j')_{1}}{4}
		\ord{X(i',j)_{1}}{1}
		+
		\ord{X(l,j)_{1}}{3}
		\ord{X(i',j')_{1}}{2}
		\right)\\
		=&X(l,j')_{k}X(i',j)_{k+1}
		+X(l,j)_{k}X(i',j')_{k+1}\\
		&+X(i',j')_{k}X(l,j)_{k+1}
		+X(i',j)_{k}X(l,j')_{k+1}+\cdots_{\sh}
	\end{align*}
	and the same element $B$ as above:
	\begin{align*}
		&B=\partial^{[2k-1]}
		\left( \ord{X(i',j')_{1}}{2}\ord{X(i,j)_{1}}{3} +\ord{X(i',j)_{1}}{1}\sum_{i\leq s<j}\ord{X(s,s)_{1}}{\downarrow}\right)\\
		&=
		X(i,i)_{k}X(i',j)_{k+1}
		+\sum_{i<s<j}X(s,s)_{k}X(i',j)_{k+1}
		+X(i,j)_{k}X(i',j')_{k+1}\\
		&+X(i',j')_{k}X(i,j)_{k+1}
		+X(i',j)_{k}\sum_{i\leq s<j}X(s,s)_{k+1}
		+\cdots_{\sh}.
	\end{align*}
	Now, we see that:
	\begin{align*}
		S(A,B)=
		&-X(l,j')_{k}\sum_{i<s<j}X(s,s)_{k}X(i',j)_{k+1}\\
		&-X(l,j')_{k}X(i,j)_{k}X(i',j')_{k+1}\\
		&+X(l,j)_{k}X(i,i)_{k}X(i',j')_{k+1}\\
		&+X(i',j')_{k}X(i,i)_{k}X(l,j)_{k+1}\\
		&+X(i',j)_{k}X(i,i)_{k}X(l,j')_{k+1}\\
		&-X(i',j')_{k}X(l,j')_{k}X(i,j)_{k+1}\\
		&-X(i',j)_{k}X(l,j')_{k}\sum_{i\leq s<j}X(s,s)_{k+1}+\cdots_{\sh},
	\end{align*}
	where the elements appear in the decreasing order. Indeed, 
	this ordering  (except among second and third terms) can be quickly seen by comparing the factors of weight $k+1$. For the second and third terms, note that:
	\begin{align*}
		X(l,j)_{k}\succ X(l,j')_{k} \succ X(i,j)_{k} \succ X(i,i)_{k}.
	\end{align*}
	Our aim is to reduce this $S$-polynomial so that the third term is the largest in the reduction. Thus, we must only reduce the first and the second terms.
	We reduce the highest term by starting from an element of type \eqref{eqn:nondegnohe}:
	\begin{align*}
		&\left(\partial^{[2k-1]}(
		\ord{X(l,j')_{1}}{4}\ord{X(i',j)_{1}}{1}
		+\ord{X(l,j)_{1}}{3}\ord{X(i',j')_{1}}{2})\right)
		\left(\sum_{i<s<j}X(s,s)_{k}\right)\\
		&=X(l,j')_{k}\left(\sum_{i<s<j}X(s,s)_{k}\right)X(i',j)_{k+1}
		+X(l,j)_{k}\left(\sum_{i<s<j}X(s,s)_{k}\right)X(i',j')_{k+1}\\
		&+X(i',j')_{k}\left(\sum_{i<s<j}X(s,s)_{k}\right)X(l,j)_{k+1}
		+X(i',j)_{k}\left(\sum_{i<s<j}X(s,s)_{k}\right)X(l,j')_{k+1}\\
		&+\cdots_{\sh}.
	\end{align*}
	Now, the leading monomial here is the same as the highest monomial in the $S$-polynomial, and rest of the terms are strictly  lower than the third term in the $S$-polynomial.
	Thus, we can reduce the $S$-polynomial to:
	\begin{align*}
		S(A,B)\xrightarrow{G_1}
		&-X(l,j')_{k}X(i,j)_{k}X(i',j')_{k+1}
		+X(l,j)_{k}X(i,i)_{k}X(i',j')_{k+1}+\cdots.
	\end{align*}		
	We now reduce the first term on the right.
	For this, we use  an element of type \eqref{eqn:semidegvert}:
	\begin{align*}
		&\left(\partial^{[2k-1]}\bigg(\ord{X(i',j')_{1}}{1}\ord{X(l,j')_{1}}{2}\bigg)\right) X(i,j)_{k}\\
		&=X(l,j')_{k}X(i,j)_{k}X(i',j')_{k+1}+X(i',j')_{k}X(i,j)_{k}X(l,j')_{k+1}+\cdots_{\sh}.
	\end{align*}		
	Since $X(i',j')_{k+1}\succ X(l,j')_{k+1} $,  reducing the $S$-polynomial further gives us:
	\begin{align}
		S(A,B)\xrightarrow{G_1}X(l,j)_{k}X(i,i)_{k}X(i',j')_{k+1}+\cdots.
		\label{eqn:S3k+1l>i'red}
	\end{align}
	We enlarge $G_2$ by including this reduction.
		
	\subsection{Weight \texorpdfstring{$3k+2$}{3k+2}}
	Here, we shall choose indices $i,j,i',j',l$ as in Figure \ref{fig:3k+2}:
	\begin{align}
		1\leq j'\leq l<j=i'<i\leq n.
		\label{eqn:wt3k+2indices}
	\end{align}
	We will show that corresponding to these indices,
	\begin{align}
	X(i',j')_{k}X(j,j)_{k+1}X(i,l)_{k+1}\in\LT(\ring{I}).
	\end{align}
	Such a leading term is indivisible by any of the quadratic leading terms of $G_1$.
	\subsubsection{Weight \texorpdfstring{$3k+2$}{3k+2} and \texorpdfstring{$j'=l$}{j'=l}}
	
	We begin with $l=j'$, and consider derivatives of elements of type \eqref{eqn:semidegvert} and \eqref{eqn:nondegtrinh}:
	\begin{align*}
		A&=\partial^{[2k-1]}
		(\ord{X(i',j')_{1}}{1}\ord{X(i,j')_{1}}{2})\\
		&=X(i,j')_{k}X(i',j')_{k+1}
		+X(i',j')_{k}X(i,j')_{k+1}+\cdots_{\sh}.
	\end{align*}
	and 
	\begin{align*}
		B&=\partial^{[2k-1]}
		\left( \ord{X(i,j')_{1}}{4}\sum_{j\leq s<i}\ord{X(s,s)_{1}}{\uparrow}
		+\ord{X(i',j')_{1}}{3}\ord{X(i,j)_{1}}{2} \right)\\
		&= X(i,j')_{k}\sum_{j\leq s<i}X(s,s)_{k+1}
		+X(i',j')_{k}X(i,j)_{k+1}\\
		&+X(i,j)_{k}X(i',j')_{k+1}
		+\sum_{j\leq s<i}X(s,s)_{k}X(i,j')_{k+1}+\cdots_{\sh}.
	\end{align*}
	So, we see that:
	\begin{align*}
		S(A,B) &= -X(i,j')_{k}\left(\sum_{j< s<i}X(s,s)_{k+1}\right)X(i',j')_{k+1}\\
		& -X(i',j')_{k}X(i,j)_{k+1}X(i',j')_{k+1}\\
		&-X(i,j)_{k}X(i',j')_{k+1}X(i',j')_{k+1}\\
		&+X(i',j')_{k}X(j,j)_{k+1}X(i,j')_{k+1}\\
		&-\left(\sum_{j\leq s<i}X(s,s)_{k}\right)X(i',j')_{k+1}X(i,j')_{k+1}+\cdots_{\sh}.
	\end{align*}
	Here, as before, the terms are written in descending order.

	Our aim is to show that this $S$-polynomial can be reduced so that the leading term of the reduction is the fourth term. 
	
	We reduce the first term by employing an element of the type \eqref{eqn:semidegvert}:
	\begin{align*}
		&\partial^{[2k-1]}\left( \ord{X(i',j')_{1}}{1}\ord{X(i,j')_{1}}{2} \right)\\
		&=X(i,j')_{k}X(i',j')_{k+1}+X(i',j')_{k}X(i,j')_{k+1}+\cdots_{\sh}
	\end{align*}
	multiplied by $\left(\sum_{j< s<i}X(s,s)_{k+1}\right)$. The leading monomial of the result is the same as that of the $S$-polynomial, but all other monomials are strictly lower than the fourth monomial in the $S$-polynomial. This eliminates the first term in the $S$-polynomial.
	
	We reduce the second term by using an element of the type obtained in \eqref{eqn:deg}:
	\begin{align*}
		\left(\partial^{[2k-1]}(X(i',j')_{1}^2)\right)X(i,j)_{k+1}
		&= 2X(i',j')_{k}X(i,j)_{k+1}X(i',j')_{k+1}+\cdots_{\sh},
	\end{align*}
	and reduce the third term using
	\begin{align*}
		\left(\partial^{[2k]}(X(i',j')_{1}^2)\right)\cdot X(i,j)_{k}
		&=X(i,j)_kX(i',j')_{k+1}X(i',j')_{k+1}+\cdots_{\sh}.
	\end{align*}

	In conclusion, we have achieved:
	\begin{align}
		S(A,B)\xrightarrow{G_1} X(i',j')_{k}X(j,j)_{k+1}X(i,j')_{k+1}+\cdots.
		\label{eqn:wt3k+2l=j'red}
	\end{align}
	We enlarge $G_2$ by including this reduction.
	
	\subsubsection{Weight \texorpdfstring{$3k+2$}{3k+2} and \texorpdfstring{$j'<l$}{j'<l}}
	
	Here, our broad strategy is still the same, except that the analogue of the $A$ element is slightly more complicated.
	The restrictions \eqref{eqn:wt3k+2indices} are still in force.
	
	We start with:
	\begin{align*}
		A&=\partial^{[2k-1]}
		(\ord{X(i,j')_{1}}{4}\ord{X(i',l)_{1}}{1}+\ord{X(i',j')_{1}}{3}\ord{X(i,l)_{1}}{2})\\
		&=
		X(i,j')_{k}X(i',l)_{k+1}+X(i',j')_{k}X(i,l)_{k+1}\\
		&\quad+X(i,l)_{k}X(i',j')_{k+1}+X(i',l)_{k}X(i,j')_{k+1}+\cdots_{\sh}
	\end{align*}
	and the same element $B$ as before:
		\begin{align*}
		B&=\partial^{[2k-1]}
		\left( \ord{X(i,j')_{1}}{4}\sum_{j\leq s<i}\ord{X(s,s)_{1}}{\uparrow}
		+\ord{X(i',j')_{1}}{3}\ord{X(i,j)_{1}}{2} \right)\\
		&= X(i,j')_{k}\sum_{j\leq s<i}X(s,s)_{k+1}
		+X(i',j')_{k}X(i,j)_{k+1}\\
		&\quad +X(i,j)_{k}X(i',j')_{k+1}
		+\sum_{j\leq s<i}X(s,s)_{k}X(i,j')_{k+1}+\cdots_{\sh}.
	\end{align*}
	We see that:
	\begin{align*}
		S(A,B)
		&=-X(i,j')_{k}\left(\sum_{j< s<i}X(s,s)_{k+1}\right)X(i',l)_{k+1}\\
		&-X(i',j')_{k}X(i,j)_{k+1}X(i',l)_{k+1}\\
		&+X(i',j')_{k}X(j,j)_{k+1}X(i,l)_{k+1}\\
		&+X(i,l)_{k}X(j,j)_{k+1}X(i',j')_{k+1}\\
		&-X(i,j)_{k}X(i',l)_{k+1}X(i',j')_{k+1}\\
		&+X(i',l)_{k}X(j,j)_{k+1}X(i,j')_{k+1}\\
		&-\left(\sum_{j\leq s<i}X(s,s)_{k}\right)X(i',l)_{k+1}X(i,j')_{k+1}	+\cdots_{\sh}.
	\end{align*}
	Here, the terms are written in the descending order.
	Our aim is to reduce this so that the leading term is the third term.
	
	In order to reduce the first term, we utilize an element of $G_1$ of the type obtained in \eqref{eqn:nondegnohe}
	\begin{align*}
		&\partial^{[2k-1]}
		(\ord{X(i,j')_{1}}{4}\ord{X(i',l)_{1}}{1}+\ord{X(i',j')_{1}}{3}\ord{X(i,l)_{1}}{2})\\
		&=
		X(i,j')_{k}X(i',l)_{k+1}+X(i',j')_{k}X(i,l)_{k+1}\\
		&\quad +X(i,l)_{k}X(i',j')_{k+1}+X(i',l)_{k}X(i,j')_{k+1}+\cdots_{\sh}
	\end{align*}
	multiplied with $\sum_{j< s<i}X(s,s)_{k+1}$. The leading monomial of the result is the same as that of the $S$-polynomial and the remaining terms  are all lower than the third term in the $S$-polynomial.
	
	For the second term, we utilize an element of $G_1$ of the type obtained in \eqref{eqn:semidegvert}
	\begin{align*}
		&\partial^{[2k-1]}(\ord{X(i',j')_{1}}{2}\ord{X(i',l)_{1}}{1})\\
		&=X(i',j')_{k}X(i',l)_{k+1}+X(i',l)_{k}X(i',j')_{k+1}+\cdots_{\sh}
	\end{align*}
	multiplied with $X(i,j)_{k+1}$. The leading monomial of the resulting  element is the same the second monomial in the $S$-polynomial, but the rest of the terms are lower than the third term in the $S$-polynomial.
	
	Thus, we have reduced the $S$-polynomial to:
	\begin{align}
		S(A,B)\xrightarrow{G_1}
		X(i',j')_{k}X(j,j)_{k+1}X(i,l)_{k+1}	+\cdots
		\label{eqn:wt3k+2l>j'red}.
	\end{align}
	Finally, we enlarge $G_2$ by including this reduction.
	
	This completes the construction of $G_2$. 
	
	\subsection{Summary of cubic leading monomials of \texorpdfstring{$G_2$}{G2}}
	\label{sec:cubicsummary}
	We have now obtained as leading monomials the cubics  corresponding to the factors in Figures \ref{fig:3k+1} and \ref{fig:3k+2}.

	\section{A family of combinatorial identities due to Dousse--Konan}
	\label{sec:dk}
	
	Our aim in this section is to recall the combinatorial identities of Dousse--Konan  \cite{DouKon-slnI} (and \cite{DouKon-slnII}) and connect them with the leading monomials which we have obtained thus far. The setup used by \cite{DouKon-slnI} 
	relies on ``difference conditions'' and 
	it is quite different from what is natural from a Gr\"obner perspective. Therefore, our main task in this section is to bridge the two formalisms.
	
	\begin{defi} Let $\sC$ be a finite set of colours, and let $\ZZ_{>0}\times\sC$ be the set of coloured positive integers. 
	Let $\Delta:\sC\times \sC\rightarrow \ZZ$ be a function (called the minimum difference function). 
	
	A coloured partition is a sequence $\pi = ( \pi_1 ,\pi_2,\cdots, \pi_s)$ sometimes alternately denoted as $\pi=\pi_1+\pi_2+\cdots+\pi_s$ such that  $s\geq 0$ ($s=0$ corresponds to the empty coloured partition denoted by $\emptyset$), each $\pi_i=({k_i})_{c_i}\in  \ZZ_{>0}\times\sC$, and finally, 
	for all $1\leq i\leq s-1$, $k_i-k_{i+1}\geq \Delta(c_i,c_{i+1})$.
	
	We denote the set of such coloured partitions by $\sP(\sC,\Delta)$.
	\end{defi}
	
	Note that generally, $\Delta\geq 0$, in which case, the parts will have weakly decreasing weights. This orders the parts exactly opposite to the order used in Definition \ref{def:ourgrevlex}. However, this is merely a difference in the convention.
	
	Note also that at the beginning, we were using the notation $X_k$  where $X\in\sX_n$ was the colour and $k\in\ZZ_{>0}$ was the weight. For this section, we have switched the notation to be consistent with \cite{DouKon-slnI}. 
	
	\begin{nota}
		A contiguous subsequence $\psi$ of $\pi$ will be denoted as $\psi \subsetdots \pi$.
	\end{nota}
	We now define coloured partitions with forbidden contiguous subsequences.
	
	\begin{defi}
		Let $\sC$, $\Delta$, $\sP(\sC,\Delta)$ be as above.
		Let $\Phi\subset \sP(\sC,\Delta)$ be a subset called the forbidden patterns/partitions. We say that $\pi\in \sP(\sC,\Delta)$ avoids (or forbids) $\Phi$, if $\psi\subsetdots\pi$ implies that  $\psi\not\in\Phi$. 
		Correspondingly, we let $\sP(\sC,\Delta,\Phi)$ be the set of those partitions in $\sP(\sC,\Delta)$ which avoid $\Phi$.
	\end{defi}
	
	Note: if $\Phi=\{\}$ is empty, then there are no forbidden contiguous subsequences and consequently $\sP(\sC,\Delta,\Phi)=\sP(\sC,\Delta)$.		

	\begin{defi}
	Corresponding to $\la{sl}_n$, the set of colours used by \cite{DouKon-slnI} is:
	\begin{align}
		\sC_{n} = \{a_ib_j \,|\, 0\leq i,j\leq n-1, (i,j)\neq (0,0)\}.
		\label{eqn:DKcols}
	\end{align}
	We will think of colours $a_ib_i$ ($1\leq i\leq n-1$) as belonging to $\la{h}$.
	\end{defi}
	
	We have been using a slightly different set of colours to align better with matrices in $\la{sl}_n$, namely:
	\begin{align*}
		\sX_n=\{X(i,j) \,|\, 1\leq i,j\leq n, (i,j)\neq (n,n)\}.
		\label{eqn:Xcols}
	\end{align*}
	The conversion is as follows:
	\begin{align}
		X(i,j) = \begin{cases}
			a_{j-1}b_{i-1}&\mathrm{if}\quad i\neq j\\
			a_jb_i&\mathrm{if}\quad i=j.
		\end{cases}
	\end{align}

	It will be beneficial for us to organize the colours $a_ib_j$ accordingly:
	\begin{enumerate}
		\item If $i\neq j$, place colour $a_ib_j$ in column $i+1$ and row $j+1$.
		\item Place colour $a_ib_i$ in row $i$ and column $i$.
	\end{enumerate}
	this is consistent with our final renaming of colours using $X$ variables.

	\begin{exam}
	The arrangement is as follows for $n=4$, i.e., $\la{sl}_4$. Note carefully the locations of $a_ib_i$!
	\begin{align*}
		\begin{matrix}
			X(1,1) & X(1,2) &X(1,3) &X(1,4)\\
			X(2,1) & X(2,2) &X(2,3) &X(2,4)\\
			X(3,1) & X(3,2) &X(3,3) &X(3,4)\\
			X(4,1) & X(4,2) &X(4,3) &
		\end{matrix}
		\quad\quad 
		\begin{matrix}
			a_1b_1 & a_1b_0 &a_2b_0 &a_3b_0\\
			a_0b_1 & a_2b_2 & a_2b_1 &a_3b_1\\
			a_0b_2 & a_1b_2 & a_3b_3 &a_3b_2\\
			a_0b_3 & a_1b_3 & a_2b_3 & 
		\end{matrix}
	\end{align*}
	\end{exam}
	
	\begin{defi}
	On $\sC_{n}$, define (see \cite[Def.\ 1.9]{DouKon-slnI}):
	\begin{align}
		\Delta(a_jb_i, a_{j'}b_{i'})=
		\chi(j\geq j')-\chi(j=i=j')+\chi(i\leq i')-\chi(i=j'=i'),
		\label{eqn:DKDelta}
	\end{align}
	and then define (cf. \cite[Eqn.\ (1.9)]{DouKon-slnI}):
	\begin{align}
		\label{eqn:DKDelta1}
		&\Delta_1(a_jb_i,a_{j'}b_{i'})=
		\Delta(a_jb_i,a_{j'}b_{i'}) \notag \\
		&\qquad+\chi(j=i=j'=i')+
		\chi(j=i=i'+1\leq j')+
		\chi(j'=i'=j+1\leq i).
	\end{align}
	\end{defi}
	
	The minimal difference function $\Delta_1$ has the following crucial properties.
	
	\begin{lem}
		\label{lem:delta1props}
		The function $\Delta_1$ satisfies the following properties:
		\begin{enumerate}
			\item 
			\label{item:Delta1summands}
			In the definition \eqref{eqn:DKDelta1} of $\Delta_1$, at most one summand is non-zero.

			\item 
			\label{item:Delta1-012}
			$\Delta_1\in \{0,1,2\}$.

			\item
			\label{item:Delta1=0}
			For $x,y$ in $\sC_{n}$, $\Delta_1(x,y)=0$ iff $x$ lies strictly to the bottom-left of $y$ (in particular, $x\neq y$ and also $x$ can not be in the same horizontal or vertical line as $y$). 
			We shall denote this as:
			
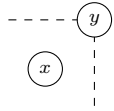
\begin{align*}
				\begin{tikzpicture}[scale=0.65]
					\node (x) at (0,0) {};
					\node (y) at (1,1) {};
					\draw[dashed] (-0.75,1) -- (y.west);
					\draw[dashed] (1,-0.75) -- (y.south);
					\node[draw,circle, inner sep=0.5ex, fill=white, minimum width=3ex] at (x)  {$\scriptstyle{x}$};
					\node[draw,circle, inner sep=0.5ex, fill=white, minimum width=3ex] at (y) {$\scriptstyle{y}$};
				\end{tikzpicture}
			\end{align*}
			
			\item 
			\label{item:Delta1xx!=0}
			For all $x\in\sC_n$, $\Delta_1(x,x)\neq 0$.

				\item \label{item:Delta1xyyx>0}
			For $x,y\in\sC_n$, we have $\Delta_1(x,y)\Delta_1(y,x)\geq 1$ iff either $x$ is to the top-left of $y$ or $y$ is to the top-left of $x$, depicted as:
			
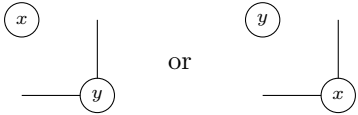
\begin{align*}
				\begin{matrix}
					\begin{tikzpicture}
						\draw (0,0) -- (1,0) -- (1,1);
						\node[draw,circle, inner sep=0.5ex, fill=white, minimum width=3ex] at (1,0)  {$\scriptstyle{y}$}; 
						\node[draw,circle, inner sep=0.5ex, fill=white, minimum width=3ex] at (0,1)  {$\scriptstyle{x}$}; 
					\end{tikzpicture}
				\end{matrix}
				\qquad\mathrm{or}\qquad
				\begin{matrix}
					\begin{tikzpicture}
						\draw (0,0) -- (1,0) -- (1,1);
						\node[draw,circle, inner sep=0.5ex, fill=white, minimum width=3ex] at (1,0)  {$\scriptstyle{x}$}; 
						\node[draw,circle, inner sep=0.5ex, fill=white, minimum width=3ex] at (0,1)  {$\scriptstyle{y}$}; 
					\end{tikzpicture}
				\end{matrix}
			\end{align*}

			\item
			\label{item:Delta1=2}
			For $x,y$ in $\sC_{n}$, $\Delta_1(x,y)=2$ iff $x$ and $y$ satisfy one of the following.
			\begin{enumerate}
				\item Neither $x$ nor $y$ are on $\la{h}$ and $x$ is to the top-right of $y$ (including $x=y$, $x$ directly above in the same column as $y$ and directly to the right in the same row as $y$)
				\item If $x$ is in $\la{h}$ and $y$ is not in $\la{h}$, then $x$ is to the top-right of $y$ and not in the same row as  $y$.
				\item If $x$ is not in $\la{h}$ but $y$ is in $\la{h}$, then $x$ is to the top-right of $y$ and is not  in the same column as $y$.
			\end{enumerate}
			We shall denote these as respectively:
			
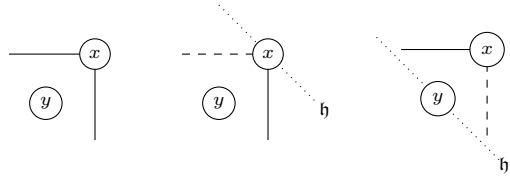
\begin{align*}
					\begin{matrix}
					\begin{tikzpicture}[scale=0.65]
						\draw[white] (-2,2)--(1.5,-1.5);
						\node (x) at (0,0) {};
						\node (y) at (1,1) {};
						\draw (-0.75,1) -- (y.west);
						\draw (1,-0.75) -- (y.south);
						\node[draw,circle, inner sep=0.5ex, fill=white, minimum width=1ex] at (x)  {$\scriptstyle{y}$};
						\node[draw,circle, inner sep=0.5ex, fill=white, minimum width=1ex] at (y) {$\scriptstyle{x}$};
					\end{tikzpicture}
					\end{matrix}
					\begin{matrix}
					\begin{tikzpicture}[scale=0.65]
						\draw[white] (-2,2)--(1.5,-1.5);
						\draw[dotted] (0,2) -- (2,0);
						\node at (2.15,-0.15) {$\scriptstyle{\la{h}}$};
						\node (x) at (0,0) {};
						\node (y) at (1,1) {};
						\draw[dashed] (-0.75,1) -- (y.west);
						\draw (1,-0.75) -- (y.south);
						\node[draw,circle, inner sep=0.5ex, fill=white, minimum width=1ex] at (x)  {$\scriptstyle{y}$};
						\node[draw,circle, inner sep=0.5ex, fill=white, minimum width=1ex] at (y) {$\scriptstyle{x}$};
					\end{tikzpicture}
					\end{matrix}
					\begin{matrix}
					\begin{tikzpicture}[scale=0.65]
						\draw[white] (-2,2)--(1.5,-1.5);
						\draw[dotted] (-1.25,1.25) -- (1.25,-1.25);
						\node at (1.35,-1.35) {$\scriptstyle{\la{h}}$};
						\node (x) at (0,0) {};
						\node (y) at (1,1) {};
						\draw (-0.75,1) -- (y.west);
						\draw[dashed] (1,-0.75) -- (y.south);
						\node[draw,circle, inner sep=0.5ex, fill=white, minimum width=3ex] at (x)  {$\scriptstyle{y}$};
						\node[draw,circle, inner sep=0.5ex, fill=white, minimum width=3ex] at (y) {$\scriptstyle{x}$};
					\end{tikzpicture}
					\end{matrix}
			\end{align*}

			\item 
			\label{item:Delta1xy=0yx=2}
			For $x,y\in\sC_n$, if $\Delta_1(x,y)=0$ then $\Delta_1(y,x)=2$.

			\item
			\label{item:Delta1tri}
			For all $x,y,z\in\sC_{n}$, $\Delta_1$ satisfies the triangle inequality (cf. \cite[Property 1.13]{DouKon-slnI} for $\Delta$):
			\begin{align}
				\Delta_1(x,y)+\Delta_1(y,z)\geq \Delta_1(x,z).
				\label{eqn:Delta1tri}
			\end{align}
		\end{enumerate}
	\end{lem}
	\begin{proof}
		Property \eqref{item:Delta1summands} can be checked directly by assuming one-by-one that one of the three $\chi$ functions in the  \eqref{eqn:DKDelta1} is non-zero. 
		
	    Property \eqref{item:Delta1-012} now follows since the same is true for $\Delta$ by \cite[Property 1.10]{DouKon-slnI} (this is straight-forward to check), and $\Delta_1$ only modifies certain $0$ values of $\Delta$ to be $1$. 
	    
		For \eqref{item:Delta1=0}, it is beneficial to create four cases, based on whether $x\in\la{h}$ and $y\in\la{h}$. We show one of the cases, namely, when $x\in\la{h}$ but $y\not\in\la{h}$. The rest of the cases are similar.		
		Let $x=a_ib_i$, $y=a_{j'}b_{i'}$ $(j'\neq i')$. Note that $x$ will be placed in row $i$ and column $i$, but $y$ will be placed in row $i'+1$ and column $j'+1$.
		Here, $\Delta_1(x,y)$ evaluates to:
		\begin{align*}
			\Delta_1(x,y)
			&=\chi(i> j')+ \chi(i\leq i') +\chi(i=i'+1\leq j').
		\end{align*}
		Therefore, 
		\begin{align*}
			\Delta_1&(x,y)=0 \iff (i'<i\leq j') \wedge (i\neq i'+1 \vee i'\geq j')\\
			&\iff (i'+1<i\leq j')  \vee (j' \leq i'<i\leq j')
			\iff (i'+1<i< j'+1).
		\end{align*}
		Writing in terms of rows and columns, the last statement is true iff $x$  is strictly to the bottom-left of $y$.
		
		Property \eqref{item:Delta1xx!=0} now follows immediately.
		
		For Property \eqref{item:Delta1xyyx>0}, to have $\Delta_1(x,y)\Delta_1(y,x)\geq 1$, neither can $x$  be strictly to the bottom-left of $y$, nor can $y$  be strictly to the bottom-left of $x$. This leaves us with the two options presented.

		The analysis for \eqref{item:Delta1=2} is very similar to that of Property \eqref{item:Delta1=0}, but we start by noting that  $\Delta_1(x,y)=2$ iff $\Delta(x,y)=2$. We omit the details.

		Property \eqref{item:Delta1xy=0yx=2} is an easy consequence of properties \eqref{item:Delta1=0} and \eqref{item:Delta1=2}.
		
		For \eqref{item:Delta1tri}, it is enough show that if the right-hand side is $1$, then the left hand side is non-zero, and if the right-hand side is $2$, then the left-hand side avoids being $0+0$, $1+0$ or $0+1$.
		
		Suppose that $\Delta_1(x,z)>0$ but $\Delta_1(x,y)=0$. This means that $x$ is \emph{not} strictly to the bottom-left of $z$ but $x$ is strictly to the bottom left of $y$. This means the situation is as follows, where the lines emanating from $y$ demarcate the position of $x$ relative to $y$, and the lines emanating from $x$ demarcate the position of $z$ with respect to $x$:
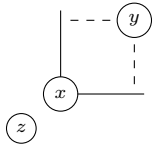
		\begin{align*}
			\begin{tikzpicture}[scale=0.65]
				\node (x) at (0,0) {};
				\node (y) at (1.5,1.5) {};
				\draw[dashed] (0.2,1.5) -- (y.west);
				\draw[dashed] (1.5,0.2) -- (y.south);
				\draw (x.north) -- (0,1.7);
				\draw (x.east) -- (1.7,0);
				\node[draw,circle, inner sep=0.5ex, fill=white, minimum width=3ex] at (x)  {$\scriptstyle{x}$};
				\node[draw,circle, inner sep=0.5ex, fill=white, minimum width=3ex] at (y) {$\scriptstyle{y}$};
				\node[draw, circle, inner sep=0.5ex, fill=white] at (-0.8,-0.7) {$\scriptstyle{z}$};
			\end{tikzpicture}
		\end{align*}
		This implies that $\Delta_1(y,z)\neq 0$ (note that $z$ could be to the top-left of $y$ or to the bottom-right of $y$, so it may happen that $\Delta_1(z,y)>0$). This proves the triangle inequality when $\Delta_1(x,z)=1$.
		
		Now suppose $\Delta_1(x,z)=2$.
		First, suppose $\Delta_1(x,y)=0$. Then, the situation is one of the following: 
		
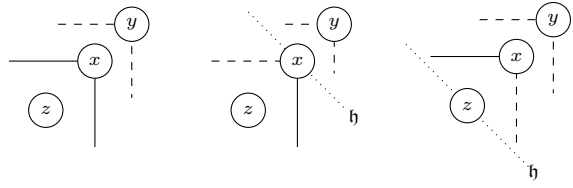
\begin{align*}
			\begin{matrix}
				\begin{tikzpicture}[scale=0.65]
					\draw[white] (-2,2)--(2,-1.5);
					\node (x) at (0,0) {};
					\node (y) at (1,1) {};
					\draw (-0.75,1) -- (y.west);
					\draw (1,-0.75) -- (y.south);
					\node[draw,circle, inner sep=0.5ex, fill=white, minimum width=3ex] at (x)  {$\scriptstyle{z}$};
					\node[draw,circle, inner sep=0.5ex, fill=white, minimum width=3ex] at (y) {$\scriptstyle{x}$};
					\draw[dashed] (0.25,1.75) -- (1.75,1.75)--(1.75,0.25);
					\node[draw,circle, inner sep=0.5ex, fill=white, minimum width=3ex] at (1.75,1.75) {$\scriptstyle{y}$};
				\end{tikzpicture}
			\end{matrix}
			\begin{matrix}
				\begin{tikzpicture}[scale=0.65]
					\draw[white] (-2,2)--(2,-1.5);
					\draw[dotted] (0,2) -- (2,0);
					\node at (2.15,-0.15) {$\scriptstyle{\la{h}}$};
					\node (x) at (0,0) {};
					\node (y) at (1,1) {};
					\draw[dashed] (-0.75,1) -- (y.west);
					\draw (1,-0.75) -- (y.south);
					\node[draw,circle, inner sep=0.5ex, fill=white, minimum width=3ex] at (x)  {$\scriptstyle{z}$};
					\node[draw,circle, inner sep=0.5ex, fill=white, minimum width=3ex] at (y) {$\scriptstyle{x}$};
					\draw[dashed] (0.75,1.75) -- (1.75,1.75)--(1.75,0.75);
					\node[draw,circle, inner sep=0.5ex, fill=white, minimum width=3ex] at (1.75,1.75) {$\scriptstyle{y}$};
				\end{tikzpicture}
			\end{matrix}
			\begin{matrix}
				\begin{tikzpicture}[scale=0.65]
					\draw[white] (-2,2)--(2,-1.5);
					\draw[dotted] (-1.25,1.25) -- (1.25,-1.25);
					\node at (1.35,-1.35) {$\scriptstyle{\la{h}}$};
					\node (x) at (0,0) {};
					\node (y) at (1,1) {};
					\draw (-0.75,1) -- (y.west);
					\draw[dashed] (1,-0.75) -- (y.south);
					\node[draw,circle, inner sep=0.5ex, fill=white, minimum width=3ex] at (x)  {$\scriptstyle{z}$};
					\node[draw,circle, inner sep=0.5ex, fill=white, minimum width=3ex] at (y) {$\scriptstyle{x}$};
					\draw[dashed] (0.25,1.75) -- (1.75,1.75)--(1.75,0.25);
					\node[draw,circle, inner sep=0.5ex, fill=white, minimum width=3ex] at (1.75,1.75) {$\scriptstyle{y}$};
				\end{tikzpicture}
			\end{matrix}
		\end{align*}
		and we see that in this case $\Delta_1(y,z)=2$.
		Similarly, if $\Delta_1(y,z)=0$, then the situation becomes:
		
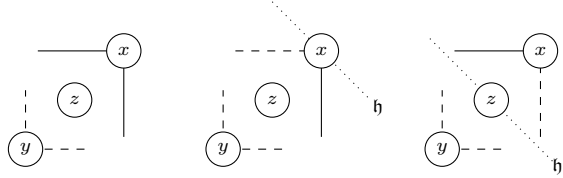
\begin{align*}
			\begin{matrix}
				\begin{tikzpicture}[scale=0.65]
					\draw[white] (-2,2)--(2,-2);
					\node (x) at (0,0) {};
					\node (y) at (1,1) {};
					\draw (-0.75,1) -- (y.west);
					\draw (1,-0.75) -- (y.south);
					\node[draw,circle, inner sep=0.5ex, fill=white, minimum width=3ex] at (x)  {$\scriptstyle{z}$};
					\node[draw,circle, inner sep=0.5ex, fill=white, minimum width=3ex] at (y) {$\scriptstyle{x}$};
					\draw[dashed] (0.2,-1) -- (-1,-1)--(-1,0.2);
					\node[draw,circle, inner sep=0.5ex, fill=white, minimum width=3ex] at (-1,-1) {$\scriptstyle{y}$};
				\end{tikzpicture}
			\end{matrix}
			\begin{matrix}
				\begin{tikzpicture}[scale=0.65]
					\draw[white] (-2,2)--(2,-2);
					\draw[dotted] (0,2) -- (2,0);
					\node at (2.15,-0.15) {$\scriptstyle{\la{h}}$};
					\node (x) at (0,0) {};
					\node (y) at (1,1) {};
					\draw[dashed] (-0.75,1) -- (y.west);
					\draw (1,-0.75) -- (y.south);
					\node[draw,circle, inner sep=0.5ex, fill=white, minimum width=3ex] at (x)  {$\scriptstyle{z}$};
					\node[draw,circle, inner sep=0.5ex, fill=white, minimum width=3ex] at (y) {$\scriptstyle{x}$};
					\draw[dashed] (0.2,-1) -- (-1,-1)--(-1,0.2);
					\node[draw,circle, inner sep=0.5ex, fill=white, minimum width=3ex] at (-1,-1) {$\scriptstyle{y}$};
				\end{tikzpicture}
			\end{matrix}
			\begin{matrix}
				\begin{tikzpicture}[scale=0.65]
					\draw[white] (-2,2)--(2,-2);
					\draw[dotted] (-1.25,1.25) -- (1.25,-1.25);
					\node at (1.35,-1.35) {$\scriptstyle{\la{h}}$};
					\node (x) at (0,0) {};
					\node (y) at (1,1) {};
					\draw (-0.75,1) -- (y.west);
					\draw[dashed] (1,-0.75) -- (y.south);
					\node[draw,circle, inner sep=0.5ex, fill=white, minimum width=3ex] at (x)  {$\scriptstyle{z}$};
					\node[draw,circle, inner sep=0.5ex, fill=white, minimum width=3ex] at (y) {$\scriptstyle{x}$};
					\draw[dashed] (0.2,-1) -- (-1,-1)--(-1,0.2);
					\node[draw,circle, inner sep=0.5ex, fill=white, minimum width=3ex] at (-1,-1) {$\scriptstyle{y}$};
				\end{tikzpicture}
			\end{matrix}
		\end{align*}
		which again implies that $\Delta_1(x,y)=2$. 
	\end{proof}

	It is now time to define the forbidden patterns to be imposed on  $\sP(\sC_n,\Delta_1)$.
	
	\begin{defi}
		Let $\Phi$ be the set of sequences:
		\begin{align*}
			\Phi&=
			\{ (k+1)_{a_{j'}b_{i'}}+k_{a_{i+1}b_{i+1}}+k_{a_{j}b_{l}}\,|\,
			0\leq i'\leq l<i=j'<j\leq n-1\}\\
			&\quad\cup
			\{ (k+1)_{a_{l}b_{i}}+(k+1)_{a_{j+1}b_{j+1}}+k_{a_{j'}b_{i'}}\,|\,
			0\leq j'\leq l<j=i'<i\leq n-1\}.
		\end{align*}
		Note that $\Phi=\{\}$ if $n=2$.
	\end{defi}
	
	Recall again that $a_{i}b_{i}$ is placed in row $i$ and column $i$, but $a_{j}b_i$ is placed in row $i+1$ and column $j+1$. 
	Therefore, based on Properties \eqref{item:Delta1=0} and \eqref{item:Delta1=2} from Lemma \ref{lem:delta1props}, it is not hard to see that $\Phi\subset\sP(\sC_n,\Delta_1)$. Thus, it gives rise to a  class of coloured partitions $\sP(\sC_n,\Delta_1,\Phi)$.	
	
	\begin{rem}
		\label{rem:Phicubic}
		Converting to $\sX_n$ colours, we see that these cubic partitions correspond \emph{exactly} to the cubic terms in Figure \ref{fig:cubic}.
	\end{rem}
	
	The following theorem is due to Dousse--Konan.
	
	\begin{thm}[{\cite{DouKon-slnI}}]
		\label{thm:dk}
		Let $F_n(q)$ be the generating function of $\sP(\sC_n,\Delta_1,\Phi)$ where $q$ keeps track of the total weight of a coloured partition. Then,
		\begin{align}
			F_n(q)=\chi_{\VOA{L}_1(\la{sl}_n)}(q).
			\label{eqn:dkchar}
		\end{align}
	\end{thm}
	\begin{rem}
		A very important remark is now in order. Actually, the theorems in \cite{DouKon-slnI} allow for an impressive amount of flexibility. They actually deal with subsets of $\sP(\sC_n,\Delta)$ (yes, we do mean $\Delta$ and not $\Delta_1$) where the sets of forbidden patterns are governed by a choice of certain functions $\delta$ and $\gamma$. As long as $\delta$, $\gamma$ satisfy conditions given in Definitions 1.17 and 1.18 of \cite{DouKon-slnI}, the generating functions of corresponding coloured partitions are equal to $\chi_{\VOA{L}_1(\la{sl}_n)}(q)$. This is the content of Theorem 1.30 of \cite{DouKon-slnI} and Theorem 1.10 of \cite{DouKon-slnII}.
		
		The particular choice of $\delta$, $\gamma$ we need does not appear explicitly in \cite{DouKon-slnI}; the closest one being the $\delta_1,\gamma_1$ that are used for Proposition 1.20 of \cite{DouKon-slnI}. Thankfully, our choice is only a small modification, as we now describe.
		
		For $k\neq l$, define
		\begin{align*}
			\delta(a_kb_l)=1+\min(k,l).
		\end{align*}
		This is the same as $\delta_1$ of \cite{DouKon-slnI}.
		We define $\gamma$ as follows.
		\begin{enumerate}
			\item If $\max(k_1,l_2)<\min(k_2,l_1)$, let:
			\begin{align}
				\gamma(a_{k_1}b_{l_1},a_{k_2}b_{l_2})=1+\max(k_1,l_2).
			\end{align}
			\item If $k_1>l_1,k_2>l_2$ and $S:=\{l_2+1,\dots,k_2\}\backslash \{l_1+1,\dots,k_1\}\neq \emptyset$, define:
			\begin{align}
				\gamma(a_{k_1}b_{l_1},a_{k_2}b_{l_2})
				=\begin{cases}
					l_2+1 & \mathrm{if}\,\,l_2+1\in S\\
					\boldsymbol{k_1+1}
					& \mathrm{otherwise}.
				\end{cases}
			\end{align}
			\item If $k_1<l_1,k_2<l_2$ and $S:=\{k_1+1,\dots,l_1\}\backslash \{k_2+1,\dots,l_2\}\neq \emptyset$, define:
			\begin{align}
				\gamma(a_{k_1}b_{l_1},a_{k_2}b_{l_2})
				=\begin{cases}
					k_1+1 & \mathrm{if}\,\,k_1+1\in S\\
					\boldsymbol{l_2+1}
					& \mathrm{otherwise}.
				\end{cases}
			\end{align}
		\end{enumerate}
	The definition of $\gamma_1$ in \cite{DouKon-slnI} uses $k_2$  in place of $\boldsymbol{k_1+1}$ and $l_1$ in place of $\boldsymbol{l_2+1}$.  
	
	The actual class of coloured partitions which arises from a fixed choice of $\delta,\gamma$ is prescribed in Definition 1.19 of \cite{DouKon-slnI}. It is completely elementary, although somewhat tedious to check that the choice $\delta_1,\gamma_1$ used in \cite{DouKon-slnI} indeed leads to the coloured partitions of Proposition 1.20 of \cite{DouKon-slnI}. It is equally elementary to check that our choice of $\delta,\gamma$ leads to the coloured partitions $\sP(\sC_n,\Delta_1,\Phi)$. 
	\end{rem}
	
	\begin{rem}
		Theorem 1.30 of \cite{DouKon-slnI} and Theorem 1.10 of \cite{DouKon-slnII} actually state that  \eqref{eqn:dkchar} holds at the refined level of multi-variate generating functions keeping track of the Cartan weights as well. 
	\end{rem}

	\begin{defi}
		\label{def:freqarray}
		Let $\sC$ be a finite set of colours. A function $f: \ZZ_{>0}\times\sC \rightarrow \ZZ_{\geq 0}$,
		such that $f$ is non-zero only finitely many times is called a frequency array. 
		The weight of $f$ is defined as:
		\begin{align*}
			\sum_{k_x\in\ZZ_{>0}\times \sC}k\cdot f(k_{x})
		\end{align*}	
		The set of all frequency arrays is denoted as $\sF(\sC)$.
	\end{defi}

	Clearly, there exists a function:
	\begin{align}
		\phi: \sP(\sC_n,\Delta_1,\Phi)\rightarrow \sF(\sC_n)
	\end{align}
	where for each $\pi\in \sP(\sC_n,\Delta_1,\Phi)$ and $k_c\in \ZZ_{>0}\times\sC$, $(\phi(\pi))(k_c)$ counts the number of times $k_c$ appears in $\pi$.

	Our main lemma is the following, based on an argument  given by Primc in \cite{Pri-crystal}.
	
	\begin{lem}
		\label{lem:DKfreqarray}
		The map $\phi$ is a weight-preserving bijection between $\sP(\sC_n,\Delta_1,\Phi)$ and the subset  $\sD\sK_n\subset\sF(\sC_n)$ defined as follows.
		We let $f\in\sD\sK_n$ iff:
			\begin{enumerate}
			\item $f(k_x)\leq 1$ for all $k\geq 1$ and $x\in\sC_n$.
			\label{item:freqkx}
			\item $f(k_x)+f(k_y)\leq 1$ for all $k\geq 1$, $x\neq y\in\sC_n$, $\Delta_1(x,y)\Delta_1(y,x)\geq 1$.
			\label{item:freqkxky}
			\item $f((k+1)_x)+f(k_y)\leq 1$ for all $x,y\in\sC$ and $k\geq 1$ with $\Delta_1(x,y)=2$.
			\label{item:freqk+1xky}
			\item $f( g_x)+f(h_y)+f(k_z)\leq 2$ for all $g_x+h_y+k_z\in
			 \Phi$.
			 \label{item:freqcubic}
		\end{enumerate}
	\end{lem}		
	\begin{proof}
		Due to properties \eqref{item:Delta1xy=0yx=2} and \eqref{item:Delta1tri} of $\Delta_1$ from Lemma \ref{lem:delta1props}, we may choose a total order, say  $\gg$, on $\sC_n$ such that 
		\begin{align}
			\Delta_1(x,y)=0 \implies x\gg y .
			\label{eqn:Delta1impliesorder}
		\end{align}
		Note that $\Delta_1(x,y)=0$ already implies $x\neq y$ due to Property \eqref{item:Delta1=0} of Lemma \ref{lem:delta1props}.
		We extend this order to a total order on $\ZZ_{>0}\times\sC_n$ (again denoted by $\gg$) such that $(k_1)_{c_1}\gg (k_2)_{c_2}$ iff:
		\begin{enumerate}
			\item $k_1>k_2$ or,
			\item $k_1=k_2$ implies $c_1\gg c_2$.
		\end{enumerate}
		
		For convenience, we may take $\gg$ to be equivalent to the  opposite of the order $\succ$ on $\ZZ_{>0}\times \sX_n$ defined in Definition \ref{def:grevlex}; compare the conditions for $\Delta_1(x,y)=0$ captured in Property \eqref{item:Delta1=0} of Lemma \ref{lem:delta1props} to that of the order $\succ$ given in Remark \ref{rem:toprightimpliesgreater}.
		
		Property \eqref{item:Delta1-012} of Lemma \ref{lem:delta1props} implies that $\Delta_1\geq 0$ and  we see that each $\pi\in \sP(\sC_n,\Delta)$ is weakly decreasing with respect to $\gg$. Thus, each $\pi\in\sP(\sC_n,\Delta_1,\Phi)$ is uniquely determined by its frequencies, or in other words, $\phi:\sP(\sC_n,\Delta)\rightarrow \sF(\sC_n)$ is injective.

		We now show that $\phi$ lands in $\sD\sK_n$.
		Let $\pi\in \sP(\sC_n,\Delta_1,\Phi)$, and denote the frequency array of  $\pi$ by $f$, i.e., $\phi(\pi)=f$.
		
		Pick colours $x,y\in\sC_n$ ($x$ may be equal to $y$) such that $\Delta_1(x,y)\Delta_1(y,x)\geq 1$, equivalently, $\Delta_1(x,y)\neq 0$ and $\Delta_1(y,x)\neq 0$. Thus, neither $k_x+k_y\subsetdots \pi$ nor $k_y+k_x \subsetdots\pi$ for any $k\geq 1$. If these two parts appear somewhere (non-adjacently) in $\pi$, then the parts in the middle must have weight $k$ (since $\Delta_1\geq 0$), and we either have a subsequence of the form $k_x+\cdots+k_z+\cdots+k_y\subsetdots \pi$ or $k_y+\cdots+k_z+\cdots+k_z\subsetdots \pi$. However, this is again not possible due to the triangle inequality (Lemma \ref{lem:delta1props} \eqref{item:Delta1tri}).
		In conclusion, in this case, $k_x$ and $k_y$ (or two copies of $k_x$ if $x=y$) simply can not both appear in $\pi$. This establishes that $\phi(\pi)$ satisfies \eqref{item:freqkx} and \eqref{item:freqkxky}.	A similar use of triangle inequality also proves \eqref{item:freqk+1xky}.
		
		Now let $(k+1)_x+k_y+k_z\in\Phi$, where the locations of $x,y,z$ are as follows:
		
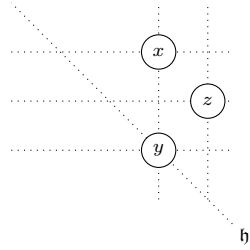
\begin{align*}
			\begin{tikzpicture}[scale=0.65]
			\draw[dotted] (4.5,-0.5) -- (0,4);
			\node at (4.75,-0.75) {$\scriptstyle{\la{h}}$};
			\draw[dotted] (0,1) -- (4.5,1);
			\draw[dotted] (0,2) -- (4.5,2);
			\draw[dotted] (0,3) -- (4.5,3);
			\draw[dotted] (3,0) -- (3,4);
			\draw[dotted] (4,0) -- (4,4);
			\node[draw,circle, inner sep=0.5ex, fill=white, minimum width=3ex] at (3,1) {$\scriptstyle{y}$};
			\node[draw,circle, inner sep=0.5ex, fill=white, minimum width=3ex] at (3,3) {$\scriptstyle{x}$};
			\node[draw,circle, inner sep=0.5ex, fill=white, minimum width=3ex] at (4,2) {$\scriptstyle{z}$};
			\end{tikzpicture}
		\end{align*}
		Suppose $(k+1)_x$ and $k_y$ are parts of $\pi$. Then, we first claim that they must appear contiguously as $(k+1)_x+k_y$. 
		If not, then there are two cases: 
		\begin{itemize}		 
		 \item We have $(k+1)_x+\cdots+k_u+\cdots+k_y\subsetdots\pi$. This implies $\Delta_1(u,y)=0$, thus, $u$ is strictly to the bottom-left of $y$, and thus, $x$ is to the top-right of $u$ (and $x,u\not\in\la{h}$). But then, $\Delta_1(x,u)=2$, contradicting the appearance of both $(k+1)_x$ and $k_u$ in $\pi$ (see the left-hand diagram in \eqref{diag:aproofofphi1}). 
		 \item We have $(k+1)_x+\cdots+(k+1)_u+\cdots+k_y \subsetdots\pi$. This implies $\Delta_1(x,u)=0$, and $x$ is strictly to the bottom-left of $u$. But then, $u$ is strictly to the top-right of $y$, and so $\Delta_1(u,y)=2$.  Therefore, both $(k+1)_u$ and $k_y$ can not appear in $\pi$ (see the middle diagram in \eqref{diag:aproofofphi1}).
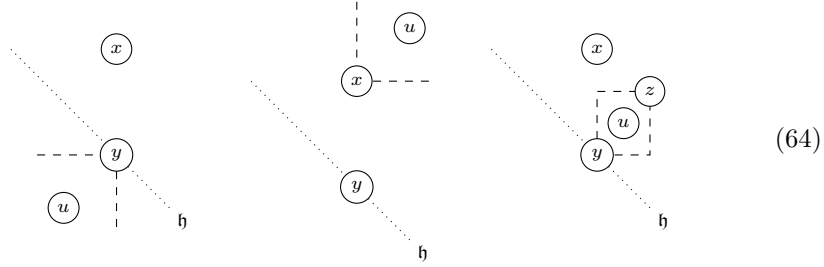
		 \begin{align}
		 	\label{diag:aproofofphi1}
		 	\begin{matrix}
		 	\begin{tikzpicture}[scale=0.7]
		 		\draw[dotted] (4,0) -- (1,3);
		 		\node at (4.25,-0.25) {$\scriptstyle{\la{h}}$};
		 		\draw[dashed] (1.5,1) -- (3,1)--(3,-0.5);
		 		\node[draw,circle, inner sep=0.5ex, fill=white, minimum width=2.5ex] at (3,1) {$\scriptstyle{y}$};
		 		\node[draw,circle, inner sep=0.5ex, fill=white, minimum width=2.5ex] at (3,3) {$\scriptstyle{x}$};
		 		\node[draw,circle, inner sep=0.5ex, fill=white, minimum width=2.5ex] at (2,0) {$\scriptstyle{u}$};
		 	\end{tikzpicture}
		 	\end{matrix}
		 	\qquad
		 	\begin{matrix}
		 	\begin{tikzpicture}[scale=0.7]
		 		\draw[dotted] (4,0) -- (1,3);
		 		\node at (4.25,-0.25) {$\scriptstyle{\la{h}}$};
		 		\draw[dashed] (3,4.5) -- (3,3)--(4.5,3);
		 		\node[draw,circle, inner sep=0.5ex, fill=white, minimum width=2.5ex] at (3,1) {$\scriptstyle{y}$};
		 		\node[draw,circle, inner sep=0.5ex, fill=white, minimum width=2.5ex] at (3,3) {$\scriptstyle{x}$};
		 		\node[draw,circle, inner sep=0.5ex, fill=white, minimum width=2.5ex] at (4,4) {$\scriptstyle{u}$};
		 	\end{tikzpicture}
		 	\end{matrix}
		 	\qquad
		 	\begin{matrix}
		 		\begin{tikzpicture}[scale=0.7]
		 			\draw[dotted] (4,0) -- (1,3);
		 			\node at (4.25,-0.25) {$\scriptstyle{\la{h}}$};
		 			\draw[dashed] (3,1) -- (3,2.2)--(4,2.2)--(4,1)--(3,1);
		 			\node[draw,circle, inner sep=0.5ex, fill=white, minimum width=2.5ex] at (3,1) {$\scriptstyle{y}$};
		 			\node[draw,circle, inner sep=0.5ex, fill=white, minimum width=2.5ex] at (3,3) {$\scriptstyle{x}$};
		 			\node[draw,circle, inner sep=0.5ex, fill=white, minimum width=2.5ex] at (4,2.2) {$\scriptstyle{z}$};
		 			\node[draw,circle, inner sep=0.5ex, fill=white, minimum width=2.5ex] at (3.5,1.6) {$\scriptstyle{u}$};
		 		\end{tikzpicture}
		 	\end{matrix}
		 \end{align}
		 
		\end{itemize}
		Next, if $k_z$ also appears in $\pi$, then it can only happen when we have another part $k_u$ with $(k+1)_x+k_y+k_u+\cdots+k_z \subsetdots\pi$ (as $(k+1)_x+k_y+k_z\in\Phi$, i.e., forbidden to appear).  Here, $y$ must be strictly to the bottom-left of $u$, which in turn must be strictly to the bottom-left of $z$. Thus, the locations are as in the right-hand figure in \eqref{diag:aproofofphi1}, but then $(k+1)_x+k_y+k_u$ is a forbidden pattern belonging to $\Phi$.
	    Combining everything, if $(k+1)_x$ and $k_y$ appear in $\pi$, $k_z$ simply can not appear anywhere in $\pi$, giving us that $f((k+1)_x)+f(k_y)+f(k_z)\leq 2$.
	    
	    The analysis for  $(k+1)_x+(k+1)_y+k_z\in\Phi$ is similar so we omit it.
	    
	    We have now proved that the image of $\phi$ is a subset of $\sD\sK_n$. 
	    
	    Lastly, we prove that our map is surjective. To this end, let $f\in\sD\sK_n$.
	    Recall that $f(k_x)\leq 1$ for all $k_x\in\ZZ_{>0}\times\sC_n$.
	    So,  we take elements $k_x$ such that $f(k_x)=1$ and arrange them  strictly decreasingly according to the total order introduced at the beginning of this proof. This is our required $\pi$, after we check that it satisfies the minimal difference conditions and forbids the patterns from $\Phi$.

	    Suppose that $k_x+k_y\subsetdots\pi$ with $x\neq y$. This implies that $x\succ y$, due to how we order parts of $\pi$.
	    Next, if $\Delta_1(x,y)>0$, we have two cases.
	    \begin{itemize}
	    	\item If $\Delta_1(y,x)>0$ then $\Delta_1(x,y)\Delta_1(y,x)>0$, but, we already have $f(k_x)+f(k_y)=2$. This conflicts with $f\in \sD\sK_n$. 
	    	\item If $\Delta_1(y,x)=0$, then we are forced to have $y\succ x$ due to the definition of our order \eqref{eqn:Delta1impliesorder} which conflicts with $x\succ y$ noted above.
	    \end{itemize}
	    Thus, in either case we get a contradiction and $\Delta_1(x,y)$ is forced to be $0$, as required.
	    The case $(k+1)_x+k_y\subsetdots\pi$ but $\Delta_1(x,y)=2$ is obviously forbidden from conditions on $f\in\sD\sK_n$, thus here also $\Delta_1(x,y)\leq 1$.
	    Lastly, it is clear that $\pi$ forbids elements of $\Phi$ as subpartitions.
	\end{proof}
	
	\section{Classical freeness of \texorpdfstring{$\ala{sl}_n$}{sl\_n} at level \texorpdfstring{$1$}{1}}
	\label{sec:main}
	We have now developed all the tools required to prove the classical freeness for $\VOA{L}_1(\la{sl}_n)$.
	\begin{thm}
		\label{thm:main}
		The vertex operator algebra $\VOA{L}_1(\la{sl}_n)$ is classically free, and the set $G_2$ is a Gr\"obner basis for the ideal $\ring{I}$.
	\end{thm}	
	\begin{proof}
		Let $\ring{L}=\{\lt(g)\,|\,g\in G_2\}$, $\ring{K}=\langle \ring{L}\rangle\subseteq \ring{I}\subseteq \ring{R}$, and consider $H_{\ring{K}}(q)$.
		This series is the generating function of a class of frequency arrays say $\sG_n\subset \sF(\sX_n)$  (recall Definition \ref{def:freqarray}) defined as follows.
		A frequency array $f$ belongs to $\sG_n$ iff:
		\begin{enumerate}
			\item If $x,y\in\sX_n$ form top-left and bottom-right corner (in any order) of a (possibly degenerate) box whose bottom-right corner is not $X(n,n)$, then $f(k_x)+f(k_y)\leq 1$.
			\item If $x,y\in\sX_n$ are the top-right and bottom-left corners (respectively) of a box as in Figure \ref{fig:iji'j'} (subject to exclusions as explained there), then $f((k+1)_x)+f(k_y)\leq 1$. 
			\item If we have a triple of colours $x,y,z$ as in Figure \ref{fig:3k+1}, then we can not have $f((k+1)_x)\geq 1, f(k_y)\geq 1, f(k_z)\geq 1$. Combined with the fact that each of the individual frequencies can be at most $1$, this can be equivalently phrased as $f((k+1)_x)+f(k_y)+f(k_z)\leq 2$.
			Similarly, if we have triples of colours $x,y,z$ as in Figure \ref{fig:3k+2}, then we can not have $f((k+1)_x)+f((k+1)_y)+f(k_z)\leq 2$.
		\end{enumerate}
		Translating to colours $\sC_n$, recalling Properties \eqref{item:Delta1xyyx>0}, \eqref{item:Delta1=2}, and the definition of $\Phi$ we see that $\sG_n$ is in weighted bijection with the set of frequency arrays $\sD\sK_n$ of Lemma \ref{lem:DKfreqarray}, which are in turn in a weighted bijection with coloured partitions $\sP(\sC_n,\Delta_1,\Phi)$. 
		The generating function of $\sP(\sC_n,\Delta_1,\Phi)$ equals $\chi_{\VOA{L}_1(\la{sl}_n)}(q)$ from Theorem \ref{thm:dk}. Thus far, we have obtained that:
		\begin{align*}
			H_{\ring{K}}(q)=\chi_{\VOA{L}_1(\la{sl}_n)}(q).
		\end{align*}
		However, combining with \eqref{eqn:chgr=chV}, \eqref{eqn:JRgrVsurj}, \eqref{eqn:JRsln1} we see:
		\begin{align}
			H_{\ring{K}}(q)=\chi_{\VOA{L}_1(\la{sl}_n)}(q)=\chi_{\gr_F(\VOA{L}_1(\la{sl}_n))}(q)\leq \chi_{J_{\infty}R_{\VOA{L}_1(\la{sl}_n)}}(q) = H_{\ring{I}}(q).
			\label{eqn:hilbineq}
		\end{align}
		Thus, the condition of Lemma \ref{lem:main} is satisfied and we  conclude that $H_{\ring{K}}(q)=H_{\ring{I}}(q)$.
		We obtain that $G_2$ is a Gr\"obner basis for $\ring{I}$ and additionally that equality holds throughout \eqref{eqn:hilbineq}. In particular, $\chi_{\gr_F\VOA{L}_1(\la{sl}_n)}(q)= \chi_{J_{\infty}R_{\VOA{L}_1(\la{sl}_n)}}(q)$, which proves classical freeness, see Theorem \ref{thm:charclassfree}.
	\end{proof}
	
	\section{Questions}
	\label{sec:questions}
	This proof raises a few important questions for further study.
	\begin{enumerate}[leftmargin=*]
		\item Here we have required two iterations of Buchberger's algorithm. Does there exist a  different total order on $\sX_n$ (but still the grevlex order globally) and/or a different basis  of $\ring{T}$ such that its derivatives already form a Gr\"obner basis of $\ring{I}$? We suspect that this is not possible, i.e, two iterations of the algorithm are (most likely) the minimum possible.
		
		\item The minimum number of iterations required in the Buchberger's algorithm (with respect to some natural grevlex order) should probably give rise to a heuristic ``degree of classical freeness'' of various VOAs.
		The VOAs $\VOA{L}_k(\la{sp}_{2n})$ ($k\in\ZZ_{\geq 0}$) are classically free in the strongest sense -- derivatives of a natural basis for the analogue of $\ring{T}$ in that case form a Gr\"obner basis. This is essentially implied by the work of Primc--\v{S}iki\'{c} and Primc--Trup\v{c}evi\'{c} \cite{PriSik-1}, \cite{PriSik-2}, \cite{PriSik-3}, \cite{PriTru}. Same holds for Virasoro minimal models at boundary admissible parameters $(2,2k+3)$ due to the work of Bruschek--Mourtada--Schepers \cite{BruMouSch-rr} and van Ekeren--Heluani \cite{vanEkeHel-chiralhom}. All these VOAs should perhaps be assigned ``degree 1'' of classical freeness.
		The VOAs $\VOA{L}_{1}(\la{sl}_n)$ $(n>2)$ will most likely have ``degree 2'' of classical freeness. Can this notion be defined rigorously?
		Besides combinatorial implications, does it say anything significant about the internal structure of VOAs?
		Does there exist a sequence of VOAs which have progressively higher degrees of classical freeness? 
		
		\item 
		A suggestion due to Andy Linshaw to connect the classical freeness of $\VOA{L}_1(\la{sl}_n)$ to the arc algebras of determinantal and permanental varieties \cite{Stu-algo} is currently being investigated. 
		
		\item In an ongoing and upcoming work with S.\ Marshall \cite{KanMar}, we shall show that the super VOAs $\VOA{L}_k(\la{osp}(1|2n))$ ($k\in\ZZ_{\geq 0}$) also have ``degree 1'' of classical freeness.
		
		\item It will be interesting to extend the approach presented here to  classical freeness of \emph{modules} for $\VOA{L}_1(\la{sl}_n)$.
		
		\item Exploration of this approach to other types (especially orthogonal types $\ala{so}_n$) and higher levels is also ongoing.
		
		\item It is known that changing the monomial order may significantly disrupt the combinatorial description and differential finiteness of the Gr\"obner basis; see the work of Afsharijoo \cite{Afs} in the case of the arc algebra of the ``fat point''  $\CC[x]/(x^k)$. It will be interesting to see what kind of combinatorial identities underlie various monomial orders for $\VOA{L}_1(\la{sl}_n)$.
	\end{enumerate}

%	\bibliographystyle{abbrv}
%	\bibliography{sln-level1-classicalfreeness}

\begin{thebibliography}{10}
		
		\bibitem{Afs}
		P.~Afsharijoo.
		\newblock Looking for a new version of {G}ordon's identities.
		\newblock {\em Ann. Comb.}, 25(3):543--571, 2021.
		
		\bibitem{AndvanEkeHel}
		G.~E. Andrews, J.~van Ekeren, and R.~Heluani.
		\newblock The singular support of the {I}sing model.
		\newblock {\em Int. Math. Res. Not. IMRN}, (10):8800--8831, 2023.
		
		\bibitem{AraLin-singular}
		T.~Arakawa and A.~R. Linshaw.
		\newblock Singular support of a vertex algebra and the arc space of its
		associated scheme.
		\newblock In M.~Gorelik, V.~Hinich, and A.~Melnikov, editors, {\em
			Representations and Nilpotent Orbits of Lie Algebraic Systems}, volume 330 of
		{\em Progress in Mathematics}, pages 1--17. Birkh{\"a}user, Cham, 2019.
		
		\bibitem{AraMor-arc}
		T.~Arakawa and A.~Moreau.
		\newblock Arc spaces and chiral symplectic cores.
		\newblock {\em Publications of the Research Institute for Mathematical
			Sciences}, 57(3):795--829, 2021.
		
		\bibitem{BruMouSch-rr}
		C.~Bruschek, H.~Mourtada, and J.~Schepers.
		\newblock Arc spaces and the {R}ogers-{R}amanujan identities.
		\newblock {\em Ramanujan J.}, 30(1):9--38, 2013.
		
		\bibitem{CalLepMil-rr}
		C.~Calinescu, J.~Lepowsky, and A.~Milas.
		\newblock Vertex-algebraic structure of the principal subspaces of certain
		{$A_1^{(1)}$}-modules. {I}. {L}evel one case.
		\newblock {\em Internat. J. Math.}, 19(1):71--92, 2008.
		
		\bibitem{CalLepMil-ag}
		C.~Calinescu, J.~Lepowsky, and A.~Milas.
		\newblock Vertex-algebraic structure of the principal subspaces of certain
		{$A^{(1)}_1$}-modules. {II}. {H}igher-level case.
		\newblock {\em J. Pure Appl. Algebra}, 212(8):1928--1950, 2008.
		
		\bibitem{Cap-id}
		S.~Capparelli.
		\newblock A construction of the level {$3$} modules for the affine {L}ie
		algebra {$A^{(2)}_2$} and a new combinatorial identity of the
		{R}ogers-{R}amanujan type.
		\newblock {\em Trans. Amer. Math. Soc.}, 348(2):481--501, 1996.
		
		\bibitem{CapLepMil-rr}
		S.~Capparelli, J.~Lepowsky, and A.~Milas.
		\newblock The {R}ogers-{R}amanujan recursion and intertwining operators.
		\newblock {\em Commun. Contemp. Math.}, 5(6):947--966, 2003.
		
		\bibitem{CapLepMil-ag}
		S.~Capparelli, J.~Lepowsky, and A.~Milas.
		\newblock The {R}ogers-{S}elberg recursions, the {G}ordon-{A}ndrews identities
		and intertwining operators.
		\newblock {\em Ramanujan J.}, 12(3):379--397, 2006.
		
		\bibitem{CMPP}
		S.~Capparelli, A.~Meurman, A.~Primc, and M.~Primc.
		\newblock New partition identities from {$C^{(1)}_\ell$}-modules.
		\newblock {\em Glas. Mat. Ser. III}, 57(77)(2):161--184, 2022.
		
		\bibitem{CoxLitOSh}
		D.~A. Cox, J.~Little, and D.~O'Shea.
		\newblock {\em Ideals, varieties, and algorithms}.
		\newblock Undergraduate Texts in Mathematics. Springer, Cham, fourth edition,
		2015.
		\newblock An introduction to computational algebraic geometry and commutative
		algebra.
		
		\bibitem{CreLinSong}
		T.~Creutzig, A.~R. Linshaw, and B.~Song.
		\newblock Classical freeness of orthosymplectic affine vertex superalgebras.
		\newblock {\em Proc. Amer. Math. Soc.}, 152(10):4087--4094, 2024.
		
		\bibitem{DouKon-sp2n}
		J.~Dousse and I.~Konan.
		\newblock Characters of level $1$ standard modules of {${C}_n^{(1)}$} as
		generating functions for generalised partitions.
		\newblock 2022.
		\newblock \url{https://arxiv.org/abs/2212.12728}.
		
		\bibitem{DouKon-slnI}
		J.~Dousse and I.~Konan.
		\newblock Generalisations of {C}apparelli's and {P}rimc's identities, {I}:
		{C}oloured {F}robenius partitions and combinatorial proofs.
		\newblock {\em Adv. Math.}, 408:Paper No. 108571, 70, 2022.
		
		\bibitem{DouKon-slnII}
		J.~Dousse and I.~Konan.
		\newblock Generalisations of {C}apparelli's and {P}rimc's identities, {II}:
		{P}erfect ${A}_n^{(1)}$ crystals and explicit character formulae.
		\newblock {\em J. Lond. Math. Soc. (2)}, 113(3):Paper No. e70469, 2026.
		
		\bibitem{FeiFeiLit-c2}
		B.~Feigin, E.~Feigin, and P.~Littelmann.
		\newblock Zhu's algebras, {$C_2$}-algebras and abelian radicals.
		\newblock {\em Journal of Algebra}, 329(1):130--146, 2011.
		
		\bibitem{Fei-pbw}
		E.~Feigin.
		\newblock The {PBW} filtration, {D}emazure modules and toroidal current
		algebras.
		\newblock {\em SIGMA Symmetry Integrability Geom. Methods Appl.}, 4:Paper 070,
		21, 2008.
		
		\bibitem{FreZhu}
		I.~B. Frenkel and Y.~Zhu.
		\newblock Vertex operator algebras associated to representations of affine and
		{V}irasoro algebras.
		\newblock {\em Duke Math. J.}, 66(1):123--168, 1992.
		
		\bibitem{GabGan}
		M.~R. Gaberdiel and T.~Gannon.
		\newblock {Z}hu's algebra, the {${C}_2$} algebra, and twisted modules.
		\newblock In {\em Vertex operator algebras and related areas}, volume 497 of
		{\em Contemp. Math.}, pages 65--78. Amer. Math. Soc., Providence, RI, 2009.
		
		\bibitem{Kan-survey}
		S.~Kanade.
		\newblock {L}epowsky--{W}ilson {$Z$}-algebras and {R}ogers--{R}amanujan-type
		identities: {R}ecent advances.
		\newblock In {\em Srinivasa Ramanujan: His Life, Legacy and Mathematical
			Influence}. Springer.
		\newblock to appear.
		
		\bibitem{KanMar}
		S.~Kanade and S.~Marshall.
		\newblock In progress.
		
		\bibitem{KKMMNN}
		S.-J. Kang, M.~Kashiwara, K.~C. Misra, T.~Miwa, T.~Nakashima, and
		A.~Nakayashiki.
		\newblock Affine crystals and vertex models.
		\newblock In {\em Infinite analysis, {P}art {A}, {B} ({K}yoto, 1991)},
		volume~16 of {\em Adv. Ser. Math. Phys.}, pages 449--484. World Sci. Publ.,
		River Edge, NJ, 1992.
		
		\bibitem{LepLi}
		J.~Lepowsky and H.~Li.
		\newblock {\em Introduction to vertex operator algebras and their
			representations}, volume 227 of {\em Progress in Mathematics}.
		\newblock Birkh\"auser Boston, Inc., Boston, MA, 2004.
		
		\bibitem{LepMil}
		J.~Lepowsky and S.~Milne.
		\newblock Lie algebraic approaches to classical partition identities.
		\newblock {\em Adv. in Math.}, 29(1):15--59, 1978.
		
		\bibitem{LepWil-pnasNewAlg}
		J.~Lepowsky and R.~L. Wilson.
		\newblock A new family of algebras underlying the {R}ogers-{R}amanujan
		identities and generalizations.
		\newblock {\em Proc. Nat. Acad. Sci. U.S.A.}, 78(12):7254--7258, 1981.
		
		\bibitem{LepWil-pnasLieRR}
		J.~Lepowsky and R.~L. Wilson.
		\newblock The {R}ogers-{R}amanujan identities: {L}ie theoretic interpretation
		and proof.
		\newblock {\em Proc. Nat. Acad. Sci. U.S.A.}, 78(2):699--701, 1981.
		
		\bibitem{LepWil-structI}
		J.~Lepowsky and R.~L. Wilson.
		\newblock The structure of standard modules. {I}. {U}niversal algebras and the
		{R}ogers-{R}amanujan identities.
		\newblock {\em Invent. Math.}, 77(2):199--290, 1984.
		
		\bibitem{Li-abelianizing}
		H.~Li.
		\newblock Abelianizing vertex algebras.
		\newblock {\em Communications in Mathematical Physics}, 259(2):391--411, 2005.
		
		\bibitem{Li-remarks}
		H.~Li.
		\newblock Some remarks on associated varieties of vertex operator
		superalgebras.
		\newblock {\em Eur. J. Math.}, 7(4):1689--1728, 2021.
		
		\bibitem{LiMil}
		H.~Li and A.~Milas.
		\newblock Jet schemes, quantum dilogarithm and {F}eigin-{S}toyanovsky's
		principal subspaces.
		\newblock {\em J. Algebra}, 640:21--58, 2024.
		
		\bibitem{LinSong-cosets}
		A.~R. Linshaw and B.~Song.
		\newblock Cosets of free field algebras via arc spaces.
		\newblock {\em Int. Math. Res. Not. IMRN}, (1):47--114, 2024.
		
		\bibitem{LinSong-invsympl}
		A.~R. Linshaw and B.~Song.
		\newblock Standard monomials and invariant theory of arc spaces {II}:
		{S}ymplectic group.
		\newblock {\em J. Algebraic Geom.}, 33(4):601--628, 2024.
		
		\bibitem{MeuPri-annfields}
		A.~Meurman and M.~Primc.
		\newblock Annihilating fields of standard modules of {${\mathfrak s}{\mathfrak
				l}(2,{\bf C})^\sim$} and combinatorial identities.
		\newblock {\em Mem. Amer. Math. Soc.}, 137(652):viii+89, 1999.
		
		\bibitem{MeuPri-sl3}
		A.~Meurman and M.~Primc.
		\newblock A basis of the basic {$\mathfrak s\mathfrak l(3,\mathbb
			C)^\sim$}-module.
		\newblock {\em Commun. Contemp. Math.}, 3(4):593--614, 2001.
		
		\bibitem{Pri-crystal}
		M.~Primc.
		\newblock Some crystal {R}ogers-{R}amanujan type identities.
		\newblock {\em Glas. Mat. Ser. III}, 34(54)(1):73--86, 1999.
		
		\bibitem{PriTru}
		M.~Primc and G.~Trup\v{c}evi\'{c}.
		\newblock Linear independence for {$C_\ell^{(1)}$} by using
		{$C_{2\ell}^{(1)}$}.
		\newblock {\em J. Algebra}, 661:341--356, 2025.
		
		\bibitem{PriSik-1}
		M.~Primc and T.~\v{S}iki\'c.
		\newblock Combinatorial bases of basic modules for affine {L}ie algebras
		{${C}_n^{(1)}$}.
		\newblock {\em J. Math. Phys.}, 57(9):091701, 19, 2016.
		
		\bibitem{PriSik-2}
		M.~Primc and T.~\v{S}iki\'c.
		\newblock Leading terms of relations for standard modules of the affine {L}ie
		algebras {${C}_n^{(1)}$}.
		\newblock {\em Ramanujan J.}, 48(3):509--543, 2019.
		
		\bibitem{PriSik-3}
		M.~Primc and T.~\v{S}iki\'{c}.
		\newblock Combinatorial relations among relations for level 2 standard
		{$C^{(1)}_n$}-modules.
		\newblock {\em J. Math. Phys.}, 64(8):Paper No. 081702, 13, 2023.
		
		\bibitem{Rus-CMPP}
		M.~C. Russell.
		\newblock Companions to the {A}ndrews-{G}ordon and {A}ndrews-{B}ressoud
		identities and recent conjectures of {C}apparelli, {M}eurman, {P}rimc, and
		{P}rimc.
		\newblock {\em SIGMA Symmetry Integrability Geom. Methods Appl.}, 22:Paper No.
		046, 2026.
		
		\bibitem{Sal-vir}
		D.~Salazar.
		\newblock Boundary minimal models and the {R}ogers-{R}amanujan identities.
		\newblock {\em J. Pure Appl. Algebra}, 230(6):Paper No. 108281, 2026.
		
		\bibitem{SongZeng-E8}
		B.~Song and X.~Zeng.
		\newblock {Z}hu's algebra and the {$C_2$}-algebra of a classically free vertex
		operator algebra.
		\newblock 2026.
		\newblock \url{https://arxiv.org/abs/2606.14407}.
		
		\bibitem{Stu-algo}
		B.~Sturmfels.
		\newblock {\em Algorithms in invariant theory}.
		\newblock Texts and Monographs in Symbolic Computation. Springer, Vienna,
		second edition, 2008.
		
		\bibitem{vanEkeHel-chiralhom}
		J.~van Ekeren and R.~Heluani.
		\newblock Chiral homology of elliptic curves and the {Z}hu algebra.
		\newblock {\em Comm. Math. Phys.}, 386(1):495--550, 2021.
		
		\bibitem{Zhu1996-modinv}
		Y.~Zhu.
		\newblock Modular invariance of characters of vertex operator algebras.
		\newblock {\em Journal of the American Mathematical Society}, 9(1):237--302,
		1996.
		
	\end{thebibliography}

\end{document}